\documentclass[a4paper,11pt,twoside,reqno]{amsart}

\usepackage[utf8]{inputenc}
\usepackage[plainpages=false,pdfpagelabels=true]{hyperref}
\usepackage{amssymb,amsthm}
\usepackage[margin=1in]{geometry}
\usepackage{slashed}
\usepackage{caption}
\usepackage[all]{xy}
\usepackage{graphicx}
\usepackage{comment}
\newcommand{\tp}{^{\mathrm{t}}}
\newcommand{\bmat}{\left(\begin{smallmatrix}}
\newcommand{\emat}{\end{smallmatrix}\right)}

\newtheorem{Satz}{Theorem}[section]

\newtheorem{Lem}[Satz]{Lemma}

\newtheorem{Thm}[Satz]{Theorem}
\newtheorem{Cor}[Satz]{Corollary}
\theoremstyle{definition}

\newtheorem{Bem}[Satz]{Remark}

\newcommand{\tr}{\operatorname{Tr}}

\newcommand{\SO}{\mathrm{SO}}

\newcommand{\ag}{\mathfrak{g}}
\newcommand{\ah}{\mathfrak{h}}
\newcommand{\ak}{\mathfrak{k}}

\newcommand{\abs}[1]{\vert #1\vert}

\newcommand{\nor}{\mathrm{nor}}

\newcommand{\Z}{\mathbb{Z}}
\allowdisplaybreaks[1]

\renewcommand{\epsilon}{\varepsilon}

\newcommand{\R}{\ensuremath{\mathbb{R}}}

\newcommand{\s}{\ensuremath{\mathbb{S}}}
\numberwithin{equation}{section}

\usepackage{color}

\newtheorem{innercustomgeneric}{\customgenericname}
\providecommand{\customgenericname}{}
\newcommand{\newcustomtheorem}[2]{%
  \newenvironment{#1}[1]
  {%
   \renewcommand\customgenericname{#2}%
   \renewcommand\theinnercustomgeneric{##1}%
   \innercustomgeneric
  }
  {\endinnercustomgeneric}
}

\newcustomtheorem{customthm}{Theorem}
\newcustomtheorem{customcor}{Corollary}

\title{Construction of harmonic maps between cohomogeneity one manifolds}

\author{Anna Siffert}
\address{Universität M\"unster, Mathematisches Institut\\
Einsteinstr. 62\\
48149 M\" unster\\
Germany}
\email{asiffert@uni-muenster.de}

\thanks{Anna Siffert thanks the German Research Foundation (DFG) for funding - Project-ID 427320536 - SFB 1442.
This project is part of Anna Siffert's project on the construction of harmonic maps between cohomogeneity one manifolds 
from the SFB-funding period
2020-2024.
}

\begin{document}
\begin{abstract}
We construct equivariant harmonic maps between cohomogeneity one manifolds.
\end{abstract} 
\maketitle

\section{Introduction}
The study of harmonic maps between Riemannian manifolds and their generalizations constitutes a vibrant research area within the field of geometric analysis, as evidenced by recent publications \cite{MR4927650, BS1, S26} and the relevant references therein.
See in addition to the fantastic book \cite{MR2044031} for more background on harmonic maps.

\smallskip

Recall that a smooth map $\varphi:(M,g)\rightarrow (N,h)$ between Riemannian manifolds is called harmonic if it is a critical point of the energy functional. 
In now we explain the definition of harmonic maps in more detail, where we closely follow \cite{MR2389639, S26} and sometimes copy sentences almost word for word.
In what follows we denote by $\Omega\subset M$ a domain with piecewise $C^1$-boundary $\partial\Omega$.
The energy integral of the map $\varphi$ over the domain $\Omega$ is defined as
\begin{align}
\label{energy}
    E_{\Omega}(\varphi)=\frac{1}{2}\int_{\Omega}\lvert d\varphi\rvert^2 v_g,
\end{align}
where $v_g$ denotes the volume form of $M$ with respect to the metric $g$.
Consider a one-parameter variation of $\varphi$, i.e.
let $\Phi:M\times (-\epsilon,\epsilon)\rightarrow N, (x,t)\mapsto \Phi(x,t)=:\varphi_t(x)$, $\epsilon\in\mathbb{R}$ with $\epsilon>0$, be a smooth map with $\Phi(x,0)=\varphi(x)$. 
The map $\Phi$ is called supported in $\Omega$, if $\varphi_t=\varphi$ on $M\setminus\Omega^{\circ}$, where $\Omega^{\circ}$ denotes the interior of $\Omega$.
A smooth map $\varphi:(M,g)\rightarrow (N,h)$ is called harmonic map if it is critical point of (\ref{energy}), this means that for all compact domains $\Omega$ and all smooth one-parameter variations $\Phi$ of $\varphi$ supported in $\Omega$, we have
\begin{align*}
    \frac{d}{dt}E_{\Omega}(\varphi_t)_{\lvert t=0}=0.
\end{align*}
In other words the first variation of the energy functional vanishes.

\smallskip

One verifies by straightforward computations that the first variation of the energy functional is given by
\begin{align}
\label{fv}
    \frac{d}{dt}E_{\Omega}(\varphi_t)_{\lvert t=0}=-\int_{M}\langle \tau(\varphi),\frac{\partial\varphi_t}{\partial t}_{\lvert t=0}\rangle v_g,
\end{align}
where $\langle\,\cdot\,,\,\cdot\,\rangle$ denotes the inner product on $\varphi^{-1}TN$ which is induced from $h$.
Further, the tension field of $\varphi$, denoted by $\tau(\varphi)$, is defined by
\begin{align*}
 \tau(\varphi)=\sum_{i=1}^{\dim M}({}^{\varphi}\nabla_{e_i}(d\varphi\, e_i)-d\varphi(\nabla_{e_i}e_i)),   
\end{align*}
where $\nabla$ denotes the Levi-Civita connection of $(M,g)$ and $(e_i)_{i=1}^{\dim M}$ is a local orthonormal frame on $M$.
Furthermore, we denote by ${}^{\varphi}\nabla$ the pull-back of the Levi-Civita connection of $(N,h)$ to the bundle $\varphi^{-1}TN$.

\smallskip

From (\ref{fv}) we obtain that the map $\varphi$ is harmonic if and only if 
\begin{align}
\label{tension}
    \tau(\varphi)=0.
\end{align}
In local coordinates the identity (\ref{tension}) is 
a semi-linear elliptic system of partial differential equations of second order.
Since no general solution theory exists for these types of differential equation systems, solving the equation (\ref{tension}) is usually difficult.
One widely used tactic for constructing solutions is the imposition of symmetry conditions. These can simplify the original problem, and in some cases make it (partially) solvable. 
Note however, that it is often difficult to find the appropriate symmetry conditions that allow for reducing the original problem to one that is (partially) solvable.

\smallskip

It has proven to be a good choice for the symmetry conditions to study equivariant maps between cohomogeneity one manifolds with respect to harmonicity. The investigation of equivariant harmonic maps between cohomogeneity one manifolds, which was started by Urakawa \cite{MR1214054}, has been significantly broadened by P\"uttmann and Siffert in \cite{MR4000241}.
With this Ansatz, harmonic map examples between cohomogeneity one manifolds have been provided for certain choices of these cohomogeneity one manifolds, see e.g. \cite{MR1436833, MR4000241, MR3427685, MR3745872, MR1214054}.

\smallskip

Recall that if $M$ supports a smooth isometric action $G\times M\rightarrow M$ of a compact Lie group $G$ such that the orbit space $M/G$ is one-dimensional, then $M$ is called a cohomogeneity one manifold. There are two types of orbits: regular and non-regular.
Regular orbits (known as principal orbits as well) have a codimension of one, and their collective union forms an open dense set $M_0$ within $M$. Orbits of singular type are non-regular orbits whose codimension is strictly greater than one.
Boundary points of $M/G$ are associated with non-regular orbits. In particular, the non-regular orbits correspond to the fibers of the boundary points of \(M/G\) via the natural projection map \(\pi: M\rightarrow M/G\). Recall further that equivariant maps are structure preserving maps between cohomogeneity one manifolds, i.e. they map orbits to orbits.

\smallskip

For equivariant maps $\varphi$ between cohomogeneity one manifolds, equation (\ref{tension}) simplifies to a singular boundary value problem (along with potentially an algebraic condition). 
Specifically, one acquires a second-order ordinary differential equation of the form \begin{align} \label{ode0} \ddot r(t)+h_1(t)\dot r(t)+h_2(t,r(t))=0, \end{align} with $r:(M/G)^{\circ}\rightarrow\mathbb{R}$ being a smooth function. 
Moreover, $h_1:(M/G)^{\circ}\rightarrow\mathbb{R},\, h_2:(M/G)^{\circ}\times\mathbb{R}\rightarrow\mathbb{R}$ are smooth functions and $(M/G)^{\circ}$ denotes the interior of $M/G$.
For $\varphi$ to be harmonic, a solution $r$ of (\ref{ode0}) must extend smoothly to $M/G$ and meet specific boundary conditions at the boundary points of $M/G$.
At the boundary points of $M/G$ that correspond to singular orbits, the differential equation (\ref{ode0}) is singular. At points of $M/G$ where the fibres under the natural projection $\pi:M\rightarrow M/G$ are not singular orbits, the differential equation (\ref{ode0}) is regular.
In \cite{S26}, the initial value problem at a singular boundary point was recently solved: It was demonstrated that if \(0 \in M/G\) is a boundary point, then for every \(v\), there exists a unique smooth solution \(r:(M/G)^{\circ}\cup\{0\}\rightarrow\mathbb{R}\) of (\ref{ode0}) with \(r(0)=0\) and \(\dot r(0)=v\), which depends continuously on \(v\). 

\smallskip

The concept of generating harmonic maps by examining deformations of metrics has been explored in numerous manuscripts, including \cite{MR716320,MR4927650, MR1122903,MR653945,MR987757} and the references cited within.
One of the most prominent problems in this area is the rendering problem, i.e. given smooth, compact Riemannian manifolds $(M,g)$ and $(N,h)$, and a homotopy class $\mathcal{H}$ of smooth maps $M\rightarrow N$, does there exist a harmonic representative in $\mathcal{H}$? This problem has been fully solved by Eells and Ferreira in \cite{MR1122903} for $\dim M\geq 3$. Namely these authors proved that if $\dim M\geq 3$, then there exists a smooth metric $\tilde{g}$ conformally equivalent to $g$, and a map $\psi\in\mathcal{H}$ such that $\psi:(M,\tilde{g})\rightarrow (N,h)$ is harmonic.

\smallskip

In the first part of this manuscript, we deal with the construction of harmonic self-maps of spheres. The aforementioned result by Eeles and Ferreira \cite{MR1122903} implies that we can always deform the domain metric in such a way that a harmonic self-map of $\s^n$, $n\geq 3$, exists.
It is however often difficult to prove the existence of harmonic maps between given manifolds if one does not allow deformations of the metrics. In the first part of this manuscript we establish the existence of several harmonic self-maps of round spheres. Namely we prove:

\begin{Thm}
\label{thm11}
For $n\in\{3,4,5,7\}$, there exist two harmonic self-maps of  $\s^n$.
\end{Thm}

Note that the statement of this theorem is not new, as harmonic self-maps have already been constructed between the spheres listed in Theorem\,\ref{thm11}, see \cite{MR1436833, MR3427685}.  In order not to overload the introduction, we have decided not to formulate this theorem more precisely at this point, but instead to explain in detail in Section\,\ref{sec-sph} why this existence result is of interest.

\smallskip

In the second part of this manuscript we use deformations of metrics to construct harmonic maps of non-compact cohomogeneity one manifolds.
 Drawing inspiration from the rendering problem, we adapt concepts from Ratto's work \cite{MR987757} and establish the existence of harmonic maps between non-compact cohomogeneity one manifolds
\begin{align*}
 M=\mathbb{R}^{k_0+1}\times G_1/H_1\times\dots\times G_m/H_m,   
\end{align*}
where $k_0\in\mathbb{N}$ with $k_0\geq 2$ and each $G_i/H_i$ is a compact, irreducible homogeneous space of dimension $k_i$ that is not flat.
The principal orbit is 
\begin{align*}
 \s^{k_0}\times G_1/H_1\times\dots\times G_m/H_m,   
\end{align*}
the singular orbit is 
\begin{align*}
 N=\{0\}\times G_1/H_1\times\dots\times G_m/H_m.   
\end{align*}
We show the following theorem:

\begin{Thm}
\label{B}
Let $(M,g)$ be a non-compact cohomogeneity one manifold as described in (\ref{mfd-nc}, \ref{metric}).
Then, the metric $g$ can be deformed to yield a metric $\tilde{g}$ for which there is a non-trivial (i.e. not a constant map and not the identity map) equivariant harmonic map $(M,\tilde{g})\rightarrow (M,g)$.
\end{Thm}

\textbf{Organization}: 
We provide preliminaries on cohomogeneity one manifolds and harmonic maps between such manifolds in Section\,\ref{sec-pre}.
In Section\,\ref{sec-sph} we construct harmonic maps between specific spheres.
In Section\,\ref{sec-nc} we construct harmonic maps between non-compact cohomogeneity one manifolds and establish Theorem\,\ref{B}.

\section{Preliminaries}
\label{sec-pre}
In this section we provide preliminaries, namely we briefly recall 
the very basic background on cohomogeneity one manifolds, 
we include a short discussion on smoothness conditions for cohomogeneity one metrics and we provide a brief review
on harmonic maps between cohomogeneity one manifolds.
Finally, we present a brief account of harmonic self-maps on spheres.
The first three subsections are presented in a manner that very closely aligns with Section 2 of \cite{S26}.

\subsection{Cohomogeneity one manifolds}
\label{sub-coh}
In this subsection we briefly recall cohomogeneity one manifolds.
This subsection closely follows the exposition in \cite{MR4000241, S26}, a few sentences have been copied almost word for word. 

\medskip

A connected Riemannian manifold $M$ is said to have cohomogeneity one if it supports a smooth isometric action $G\times M\rightarrow M$ of
a compact Lie group $G$ such that the orbit space $M/G$ has dimension one. 
In this case, $M$ is also referred to as cohomogeneity one manifold.

\smallskip

Mostert \cite{MR95897,MR85460} established the following  classification result for cohomogeneity one manifolds, namely he showed that the orbit space $M/G$ of such manifolds is isometric to exactly one of the following spaces:
\begin{enumerate}
    \item the real line, i.e. $M/G=\mathbb{R}$;
    \item the circle, i.e. $M/G=S^1$;
    \item a closed interval, i.e. $M/G=[0,L]$ for some $L\in\R_+$; 
    \item a half line, i.e. $M/G=[0,\infty)$.
\end{enumerate}
From now on, we will refer to these four possibilities as cases (1), (2), (3), and (4), respectively.
By $\partial(M/G)$ we denote the boundary points of the orbit space $M/G$, i.e. $\partial(\mathbb{R})=\emptyset$, $\partial(S^1)=\emptyset$, $\partial([0,L])=\{0,L\}$ and $\partial([0,\infty))=\{0\}$. Furthermore, $(M/G)^{\circ}$ denotes the interior of the orbit space $M/G$.

\smallskip

Below let $M$ be  a cohomogeneity one manifold.
Denote by $\pi:M\rightarrow M/G$ the natural projection map.
The fibre of a boundary point of the orbit space, i.e. an element of $\partial(M/G)$, is a non-principal orbit $N$.
In case (3) there are exactly two non-principal orbits $N_0$ and $N_1$.
In case (4) there is exactly one non-principal orbit $N_0$.
 Further, the fibre of each interior point of $M/G$ is a principal orbit.

 \smallskip

There are two types of non-principal orbits, namely exceptional and singular orbits.
A non–principal orbit $N$ is called exceptional orbit, or simply exceptional, if its dimension equals the common dimension of the principal orbits.
A non–principal orbit, which is not exceptional, is called singular orbit, or simply singular.
The collection of all principal orbits is denoted by $M_0$ and is called the regular part of $M$. 

\smallskip

Let $\gamma$ be a unit-speed normal geodesic. By this we refer to an unit speed geodesic $\gamma:\R\to M$ that passes through all orbits perpendicularly and in addition satisfies $\gamma(0)\in N_0$ and $\gamma(L) \in N_1$ in case (3), and $\gamma(0)\in N_0$ in case (4).
The isotropy groups of the regular points $\gamma(t)$, $t\in (M/G)^{\circ}$, are constant. This common principal isotropy group is henceforth denoted by $H$.

\smallskip

In case (3), i.e. when $M/G$ is isometric to a closed interval, we further assume that $\gamma$ is closed or, equivalently, that the Weyl group $W$ is finite.
Recall that the Weyl group $W$ is defined to be the subgroup of the elements of the Lie group $G$ that leave $\gamma$ invariant modulo the subgroup of elements that fix $\gamma$ pointwise. The Weyl group is a dihedral subgroup of $N(H)/H$ generated by two involutions that fix $\gamma(0)$ and $\gamma(L)$, respectively, and it acts simply transitively on the regular segments of the normal geodesic.

\smallskip

The Riemannian metric \( g \) of \( M \) is described using a one-parameter family of \( G \)-invariant metrics \( g_t \) on the principal orbits \( G/H \). Specifically, we have 
\begin{align} 
\label{metric-c} 
g=g_t+dt^2, 
\end{align} 
where \( g_t \) must meet invariance conditions in case (2) and smoothness conditions in cases (3) and (4), see e.g. \cite{MR1923478} and the references therein. 
The following subsection contains a discussion of the smoothness conditions for the metric in cases (3) and (4).
Finally, note that because $M_0$ is dense in $M$, (\ref{metric-c}) serves to describe the metric for all of $M$ as well.

\subsection{Smoothness conditions of cohomogeneity one metrics}
We address the smoothness conditions for (\ref{metric-c}) at singular orbits in this subsection.
Therefore, let us assume that $M$ is a cohomogeneity one manifold containing (at least one) singular orbit. 
Moreover, let us assume that the dense and open set $M_0\subset M$ possesses a $G$-invariant metric of the type (\ref{metric-c}).
In \cite{MR1758585} and \cite{MR4400726} conditions necessary for the extension of this metric to the singular orbit to be smooth has been examined and smoothness criteria have been provided. 
 It is often easier to apply the smoothness conditions described by Verdiani and Ziller in \cite{MR4400726} in practice than those provided in \cite{MR1758585}; see \cite{VZ1} for a comparison. The subsequent presentation closely follows the one in \cite{MR4400726}.

\smallskip

A non-compact cohomogeneity one manifold with a singular orbit can be considered as a vector bundle, a compact cohomogeneity one manifold can be considered as the union of two vector bundles, see e.g. \cite{MR1155662,MR3362465} and the references therein. 
In this subsection we consider the local problem of 
the extension of the metric (\ref{metric-c}) to a singular orbit.
Therefore, we can limit the following considerations to a single vector bundle.

\smallskip

Let $G$ be a compact Lie group, and let $H$ and $K$ be closed Lie subgroups of $G$ such that $H \subset K \subset G$ and $K/H = \s^{\ell}$ for some $\ell \in \mathbb{N} = \{1, 2, 3, \dots\}$.
This means that $K$ acts transitively on $\s^{\ell}$; in other words, there is a representation $\rho:K\rightarrow O(\ell+1)$ such that for every $x\in\s^{\ell}$, $\rho(K) x=\s^{\ell}$.
The isotropy group $K$ acts on the sphere $\s^{\ell}$, and this action can be extended to a linear action on the disk $D:=D^{\ell+1}\subset\mathbb{R}^{\ell+1}$.
Hence, $M=G\times_KD$ constitutes a homogeneous disc bundle, where the base space is $G/K$ and the structure group is $H$. The isotropy group $K$ acts on $G\times D$ in the following way: 
\begin{align*} (
k, (g,d))\mapsto (gk^{-1}, \rho(k)d).   
\end{align*}
Also, note that the action of $G$ on $M$ via left multiplication has cohomogeneity one, specifically 
$$(g,[\hat{g},d])\mapsto[g\hat{g},d]$$ for $d\in D$ and $g, \hat{g}\in G$.
The singular orbit $N:=G\times_K\{0\}$ is diffeomorphic to the space $G/K$.
For sufficiently small values of $t$, the principal orbits are represented by $P_t:=G\times_K\s^{\ell}_t$. Here, $\s^{\ell}_t\subset D$ represents the sphere with radius $t$. Every $P_t$ is a diffeomorphic copy of $G/H$.
The disk $D$ can be identified G-equivariantly with a manifold that is orthogonal to the singular orbit at the point $\gamma(0)$ using the exponential map. Thus, $D$ can be viewed as a slice of the action of $G$ on $M$.

\smallskip

To state the main result of \cite{MR4400726}, let us consider a normal geodesic $\gamma\colon[0,\infty)\to \mathbb{R}^{\ell+1}$. 
The values of the metric on $M_0$ along $\gamma$ determine it.
This is due to the fact that these values and the action of $G$ determine the metric on all of $M_0$.
Let $\mathfrak{g}$ and $\mathfrak{h}$ denote the Lie algebras corresponding to $G$ and $H$, respectively. 
Let $Q$ denote a predetermined biinvariant metric on $G$. In addition, denote the orthonormal complement of the Lie algebra $\ah$ in~$\ag$ by $\mathfrak{n}$.
The complement $\mathfrak{n}$ can be identified with the tangent space to the regular orbits along $\gamma$ through action fields, i.e. by the map  $X\in\mathfrak{n} \mapsto X^{\ast}(\gamma(t))$.

\smallskip

Let $X_i$ be a basis of $\mathfrak{n}$. Then, for every $t>0$, $X^*_i(\gamma(t))$ forms a basis of $\dot \gamma^\perp(t)\subset T_{\gamma(t)}M$.
Moreover, the metric is determined by the collection of the $r$ functions $g_{ij}(t)=g(X_i^*,X_j^*)_{c(t)},\ i\le j$.

With this preparation at hand, we can reference the following theorem that provides smoothness conditions for $g$:

\begin{Thm}[\cite{MR4400726}]
\label{gsmooth}
Let $g_{ij}(t),\ t>0$ be a smooth family of positive definite matrices describing the cohomogeneity one metric on the regular part along a normal geodesic $\gamma(t)$.     Then there exist integers $a_{ij}^k$ and $ d_k$, with $\ d_k\ge 0$,  such that the metric has a smooth extension to all of $M$ if and only if
$$\sum_{i,j} a_{ij}^k\,g_{ij}(t)=t^{d_k}\phi_k(t^2)\quad \text{ for } k=1,\cdots, r, \ \text{ and } t>0 $$
 where $\phi_1,\cdots,\phi_{r}$ are smooth functions defined for $t\ge 0$.
\end{Thm}

\subsection{Harmonic maps between cohomogeneity one manifolds}
\label{sub-har}
We will briefly recall the basics of harmonic maps between cohomogeneity-one manifolds, compare \cite{MR4000241,S26}.

\smallskip

Let $G$ and $K$ act on $M$ and $N$ with cohomogeneity one, respectively.
Moreover, let 
\begin{align*} A: G\rightarrow K \end{align*}
denote a group homomorphism. Choose a unit-speed normal geodesic $\gamma$ on $M$ and a unit-speed normal geodesic $\hat{\gamma}$ on $N$.

\smallskip

We define the equivariant map $\psi$ as follows: 
\begin{align*} 
\psi(g\cdot \gamma(t))=A(g)\cdot \hat{\gamma}(r(t)), \end{align*} 
where $r:(M/G)^{\circ}\rightarrow \mathbb{R}$ is a smooth function that meets boundary conditions in cases (3) and (4), as detailed in (i) and (ii) of Subsection\,\ref{sub-har}.
The equivariance of the tension field implies that it suffices to study the tension field of $\psi$ along $\gamma(t)$.

\smallskip

To present the normal component of the tension field in a computationally convenient way, we first need to introduce some notation:
Following \cite{MR4000241}, we let \begin{gather*} \Pi_t^{r(t)}: T_{\hat{\gamma}(t)} (K\cdot\hat{\gamma}(t)) \to T_{\hat{\gamma}(r(t))} (K\cdot \hat{\gamma}(r(t))) \end{gather*} represent the parallel transport along the normal geodesic~$\hat{\gamma}$.
Furthermore, we let $J_t^{r(t)}$ be the endomorphism of $T_{\hat{\gamma}(t)} (K\cdot\hat{\gamma}(t))$, which is obtained by composing the action field homomorphism $X^{\ast}_{\vert\hat{\gamma}(t)} \mapsto X^{\ast}_{\vert \hat{\gamma}(r(t))}$ with $(\Pi_{r(t)}^t)^{-1} = \Pi_t^{r(t)}$. 
Additionally, we define $f_0 = \dot{\hat{\gamma}}(t)$ and select $F_1,\ldots, F_m\in \ak$ so that $f_1 = F^{\ast}_{1\vert\hat{\gamma}(t)},\ldots,f_m = F^{\ast}_{m\vert\hat{\gamma}(t)}$ constitute an orthonormal basis of $T_{\hat{\gamma}(t)}(K\cdot\hat{\gamma}(t))$.

Having this preparation at hand, we can now provide the normal component of the tension field.

\begin{Thm}
The normal component $\tau^{\nor}_{\vert\gamma(t)} = \langle\tau_{\vert\gamma(t)}, \dot{\hat{\gamma}}(r(t))\rangle$ of the tension field is given by
\begin{gather*}
  \tau^{\nor}_{\vert\gamma(t)} = \ddot r(t) - \dot r(t) \tr S_{\vert \gamma(t)}
    + \sum_{\mu=1}^n\sum_{i,j=1}^mc_{\mu,i}c_{\mu,j}\langle(J_t^{r(t)})^{\ast} (\Pi_t^{r(t)})^{-1} \hat{S}_{\vert \hat{\gamma}(r(t))} \Pi_t^{r(t)} J_t^{r(t)} f_i,f_j\rangle_{\lvert\hat{\gamma}(t)},
\end{gather*}
where
\begin{gather*}
c_{\mu,i}:=c_{\mu,i}(t):=\langle (dA_{e}E_{\mu})^{\ast}, f_i\rangle_{\lvert\hat{\gamma}(t)}.
\end{gather*}
Here $e$ denotes the identity element of $G$.
Moreover, $S_{\vert \gamma(t)}$ denotes the shape operator of the orbit $G\cdot \gamma(t)$ at $\gamma(t)$, $\hat{S}_{\vert \hat{\gamma}(t)}$ denotes the shape operator of the orbit $K\cdot \hat{\gamma}(t)$ at $\hat{\gamma}(t)$ and $(J_t^{r(t)})^{\ast}$ denotes the adjoint endomorphism of $J_t^{r(t)}$.
\label{normalone}
\end{Thm}

Similar to \cite{MR4000241}, we give a second formula for the normal component of the tension field, specifically one that utilizes data from the acting groups $G$ and $K$. 
Let $\hat{Q}$ denote a fixed biinvariant metric defined on $K$. Let $\hat{\mathfrak{n}}$ be the orthonormal complement of the Lie algebra of the principal isotropy group in $\ak$. 
Define the metric endomorphisms $\hat{P}_t : \hat{\mathfrak{n}} \to \hat{\mathfrak{n}}$ by \begin{gather*} \hat{Q}(X, \hat{P}_t\cdot Y) = \langle X^{\ast},Y^{\ast} \rangle_{\vert\hat{\gamma}(t)}. \end{gather*}

We have
\begin{multline*}
  \langle X^{\ast},\hat{S}\cdot X^{\ast}\rangle_{\vert\hat{\gamma}(r(t))}=  -\langle X^{\ast}, \nabla_{\hat{T}} X^{\ast}\rangle_{\vert\hat{\gamma}(r(t))}= -\tfrac{1}{2\dot r(t)} \tfrac{d}{dt} \langle X^{\ast},X^{\ast}\rangle_{\vert\hat{\gamma}(r(t))} \\
  = -\tfrac{1}{2} \hat{Q}(X,(\dot{\hat{P}})_{r(t)} X)
  = -\tfrac{1}{2} \langle X^{\ast}, (\hat{P}_t^{-1}(\dot{ \hat{P}})_{r(t)} X)^{\ast} \rangle_{\vert\hat{\gamma}(t)}
\end{multline*}
and hence the following statement holds.

\begin{Thm}
The normal component of the tension field is given by
\begin{gather*}
  \tau^{\nor}_{\vert\gamma(t)} = \ddot r(t) + \tfrac{1}{2}\dot r(t) \tr P_t^{-1}\dot P_t
    - \tfrac{1}{2}\sum_{\mu=1}^n\sum_{i,j=1}^mc_{\mu,i}c_{\mu,j}P_{i,j},
\end{gather*}
where \begin{gather*}
c_{\mu,i}:=c_{\mu,i}(t):=\langle (dA_{e}E_{\mu})^{\ast}, f_i\rangle_{\lvert\hat{\gamma}(t)}.
\end{gather*}
and $P_{i,j}$ is the $(i,j)$-entry of the representing matrix of $\hat{P}_t^{-1}(\dot{ \hat{P}})_{r(t)}$.
\label{normaltwo}
\end{Thm}

The tangential component of the tension field is given as follows:

\begin{Thm}
\label{tanaltpart}
The tangential component of the tension field is given by
\begin{gather*}
  \tau^{\tan}_{\vert\gamma(t)} = 
  \bigl(\hat{P}_{r(t)}^{-1} \sum_{\mu=1}^n [dA_{e}E_{\mu},\hat{P}_{r(t)}dA_{e}E_{\mu}]\bigr)^{\ast}_{\vert\hat{\gamma}(r(t))}
\end{gather*}
where $E_1,\ldots,E_m \in \mathfrak{n}$ are such that $E^{\ast}_{1\vert\gamma(t)},\ldots, E^{\ast}_{n\vert\gamma(t)}$ form an orthonormal basis of $T_{\gamma(t)}(G\cdot{\gamma}(t))$.
\end{Thm}

Recall form (\ref{metric-c}) that the metric on a cohomogeneity one manifold is of the form 
\begin{align*}
    g=g_t+dt^2,
\end{align*}
where $g_t$ is a one-parameter family of $G$-invariant metrics $g_t$ on the principal orbit $G/H$.
We now consider conformal changes of the metric $g_t$ of the domain metric and obtain the following Lemma, see \cite{S26}:

\begin{Lem}
Let $(M,g)$ be a cohomogeneity one manifold, with
\begin{align*}
    g=g_t+dt^2,
\end{align*}
where $g_t$ is a one-parameter family of $G$-invariant metrics $g_t$ on the principal orbit $G/H$.
Further, let $\psi:(M,dt^2+\exp(2\alpha(t))g_t)\rightarrow (M,g)$ be the equivariant map given by 
\begin{align*}
\psi(g\cdot \gamma(t))=g\cdot {\gamma}(r(t)),
\end{align*}
where $r:(M/G)^{\circ}\rightarrow \mathbb{R}$ which in addition satisfies boundary conditions in cases (3) and (4).
Then we have
\begin{gather*}
  \tau^{\nor}_{\vert\gamma(t)} = \ddot r(t) + \tfrac{1}{2}\dot r(t) \tr P_t^{-1}\dot P_t+\dim(\mathfrak{n})\dot \alpha(t)\dot r(t)
    - \tfrac{1}{2} \tr P_t^{-1} (\dot P)_{r(t)}.
\end{gather*}
\end{Lem}

\subsection{Harmonic self-maps of spheres}
This subsection presents several results regarding harmonic self-maps of spheres. In most paragraphs, the presentation of this subsection adheres closely to that in \cite{MR4000241}, some sentences are copied almost word for word.

\medskip

It is evident that we examine case (3) in this subsection.
Let $M$ be a Riemannian manifold on which a compact Lie group $G$ acts isometrically via $G\times M\to M$. Assume that the orbit space $M/G$ is isometric to the closed interval $[0,L]$ and that the action has a finite Weyl group $W$.
P\"uttmann constructed an infinite family of equivariant self-maps of $M$ by mapping $g\cdot\gamma(t)$ to $g\cdot\gamma(kt)$, as detailed in \cite{MR2480860}. 
In this context, $\gamma$ is a normal geodesic with unit speed, and $\gamma(0)$ lies within one of the non-principal orbits. 
The integer $k$ can be expressed as $j\abs{W}/2+1$, where $j\in 2\mathbb{Z}$ (and odd integers may also be permissible depending on the action). The map $g\cdot\gamma(t) \mapsto g\cdot\gamma(kt)$ is referred to as the {\em $k$-map} of $M$.
A $k$-map's degree is $k$ when the codimensions of the non-principal orbits are both odd; otherwise, it is $0$ or $\pm 1$.
With $M$ given, a natural question arises as to whether any of the $k$-maps are harmonic. More generally, we examine the equivariant {\em $(k,r)$-maps}, which are defined as the maps 
\begin{align*} 
g\cdot\gamma(t)\mapsto g\cdot\gamma(r(t)) 
\end{align*} 
where $r: [0,L] \to \mathbb{R}$ is a smooth function satisfying $r(0) = 0$ and $r(L) = kL$. 
 It is evident that any $(k,r)$-map is equivariantly homotopic to the corresponding $k$-map. A $(k,r)$-map of $M$ is harmonic if its tension field $\tau$ is zero. The tension field divides into two natural components: the component $\tau^{\tan}$ that is tangential to the orbits, and the component $\tau^{\nor}$ that is perpendicular to the orbits.

\smallskip

So far, all considerations of this subsection apply to any cohomogeneity one manifold with $M/G=[0,L]$. Starting now, we concentrate on round spheres.
Hsiang and Lawson \cite{MR298593} classified the cohomogeneity one actions on spheres. They demonstrated that every one of these actions is orbit equivalent to the isotropy representation of a rank-$2$ Riemannian symmetric space.
Any isometric cohomogeneity one action $G\times \s^{n+1} \to \s^{n+1}$ produces orbits that form an isoparametric foliation on the sphere. Takagi and Takahashi \cite{MR334094} determined the number $g$ of distinct principal curvatures of the orbits, along with their multiplicities $m_0,\ldots,m_{g-1}$. It turned out that \begin{gather*}
  m_0 = m_2 = \ldots = m_{g-2} \quad \text{and} \quad m_1 = m_3 = \ldots = m_{g-1}. 
  \end{gather*}
Specifically, $n = \frac{m_0 + m_1}{2} g$. M\"unzner \cite{MR583825} demonstrated later that this identity holds for all isoparametric foliations of spheres. Only actions with the following $(g,m_0,m_1)$ exist up to the ordering of $m_0$ and $m_1$:
\begin{multline*}
 (1,m,m), (2,m_0,m_1), (3,1,1), (3,2,2), (3,4,4), (3,8,8),\\ (4,m_0,1), (4,2,2), (4,2,2\ell+1), (4,4,4\ell+3), (4,4,5), (4,6,9), (6,1,1), (6,2,2).
\end{multline*}
According to the classification, all cohomogeneity one actions with the specified data $(g,m_0,m_1)$ are orbit equivalent, except for the case of $(4,2,1)$ where two distinct classes exist. A detailed list can be found in \cite{MR2406265}.

\smallskip

A cohomogeneity one action on a sphere characterized by the data $(g,m_0,m_1)$ is referred to as a {\em $(g,m_0,m_1)$-action}. If $m_0 = m_1 =: m$ (which, according to M\"unzner's results, must hold for odd $g$), we refer to the action as a $(g,m)$-action. 

\smallskip

As mentioned above, the tension field divides into two natural components: the component $\tau^{\tan}$ that is tangential to the orbits, and the component $\tau^{\nor}$ that is perpendicular to the orbits.
For $(g,m_0,m_1)$-actions the normal component of the tension field of a $(k,r)$-map has been provided in \cite{MR4000241}:

\begin{Thm}[Theorem\,E in \cite{MR4000241}]
Given a $(g,m_0,m_1)$-action on a sphere $\s^{n+1}$ with $n = \frac{m_0+m_1}{2}g$, the normal component of the tension field of a $(k,r)$-map vanishes if and only if $r$ satisfies the {\em $(g,m_0,m_1,k)$-boundary value problem}
\begin{multline}
\label{bvp-t}
  0 =
  4\sin^{2} gt \cdot \ddot r(t) + \bigl( g(m_0+m_1)\sin 2gt + 2g(m_0-m_1)\sin gt \bigr)\dot r(t)\\
  - g(g -2)\sin 2(r(t)-t ) \bigl( m_0+m_1 + (m_0-m_1)\cos gt \bigr)\\
  -2g \sin\bigl(2(r(t)-t )+gt\bigr) \bigl( (m_0+m_1)\cos gt +m_0-m_1\bigr)
\end{multline}
for functions $r: \,]0,\frac{\pi}{g}[\, \to \R$ with
\begin{gather*}
  \lim_{t\rightarrow 0}r(t)=0 \quad \text{and}\quad \lim_{t\rightarrow \tfrac{\pi}{g}}r(t)=k\tfrac{\pi}{g}
\end{gather*}
\end{Thm}

Note that if $m_0 = m_1 = m$ is true, the ODE in the preceding theorem simplifies to
\begin{multline}
\label{ode}
0 = 2\sin^2 gt \cdot\ddot r(t)
  + mg\sin 2gt \cdot\dot r(t) \\ -mg\big((g-1)\sin 2(r(t)-t) + \sin 2(r(t)+(g-1)t)\big).
\end{multline}
Below we refer to this ODE as \textit{(g,m)-ODE}.

\smallskip

For $(g,m_0,m_1)$-actions the following result for the tangential component of the tension field of a $(k,r)$-map has been provided in \cite{MR4000241}:

\begin{Thm}[see Theorem\,H in \cite{MR4000241}]
Given a $(g,m_0,m_1)$-action on $\s^{n+1}$ the tangential component of the tension field of any $(k,r)$-map vanishes except possibly for
\begin{gather*}
  (g,m_0,m_1) \in \{ (4,2,2\ell+1), (4,4,4\ell+3), (4,4,5), (4,6,9) \}.
\end{gather*}
\end{Thm}

For $g=1$ and $2\leq m\leq 5$, infinitely many solutions of the $(1,m,\pm1)$-boundary value problem have been provided in \cite{MR1436833}.
For $g=2$ and $2\leq m_0\leq 5$, $m_0\leq m_1$, infinitely many solutions of the $(2,m_0,m_1,\pm1)$-boundary value problem have been provided in \cite{MR3427685}.

\section{Construction of harmonic maps between spheres}
\label{sec-sph}
In this section we construct several harmonic self-maps of round spheres. Throughout we are dealing with the case $m_0=m_1=m$ only.

\subsection{Preparations} 
This subsection gathers various preparations necessary for later proving the existence of multiple harmonic self-maps of round spheres.

\subsubsection{The variable $x$}
For the subsequent considerations, it is useful to extend the ends of the interval $(0,\tfrac{\pi}{g})$ to $-\infty$ and $\infty$, respectively.
We achieve this by defining the variable $x:=\log(\tan \tfrac{gt}{2})$.
In terms of $x$, the $(g,m)$-boundary value problem is given by \begin{multline} \label{bvp-x} r''(x)=(m-1)\tanh x\,r'(x)\\+\tfrac{m}{2g}\left[(g\!-\!1)\sin\big(2r(x)\!-\!2a(x)\big)+\sin\big(2r(x)\!+\!2(g\!-\!1)a(x)\big)\right]
\end{multline}
with $\lim_{x\rightarrow -\infty}r(x)=0$ and
$\lim_{x\rightarrow\infty}r(x)=\tfrac{(kg\!+\!1)\pi}{g}$, $k\in\Z$,
where
\begin{align*}
a:\R\rightarrow\R,\quad x\mapsto\tfrac{2}{g}\arctan\exp x.
\end{align*} 
The linear solutions $r(t)=t$ and $r(t)=-(g-1)t$ of the boundary value problem (\ref{bvp-t}) correspond to the solutions $r(x)=a(x)$ and $r(x)=-(g-1)a(x)$ in the boundary value problem (\ref{bvp-x}), respectively.
We also call the latter \lq linear solutions\rq\, of (\ref{bvp-x}) by an abuse of notation. 

\subsubsection{The function $W$.}
\label{wr}
\label{hi} 
Throughout this paper we make use of the function $W$, which plays the role of a Lynapunov function.

\smallskip

For any solution $r$ of the $(g,m)$-ODE define $W:\R\rightarrow\R$ by
\begin{align*}
W(x)=\tfrac{1}{2}r'(x)^2&-\tfrac{m}{2g}\big((g\!-\!1)\cos 2a(x)+\cos2(g\!-\!1)a(x)\big)\sin^2r(x)\\
&-\tfrac{m}{2g}\big(-(g\!-\!1)\sin2a(x)+\sin2(g\!-\!1)a(x)\big)\sin^2\big(r(x)\!-\!\tfrac{3\pi}{4}\big).
\end{align*}
Plugging in the $(g,m)$-ODE yields
\begin{align*}
W'(x):=\tfrac{d}{dx}W(x)=&(m\!-\!1)\tanh(x)r'(x)^2\\&+\tfrac{m(g\!-\!1)}{g^2\cosh x}\big(\sin2a(x)+\sin2(g\!-\!1)a(x)\big)\sin^2r(x)\\
&-\tfrac{m(g\!-\!1)}{g^2\cosh x}\big(-\cos2a(x)+\cos2(g\!-\!1)a(x)\big)\sin^2\big(r(x)\!-\!\tfrac{3\pi}{4}\big).
\end{align*}

It is convenient to introduce $h_i:\R\rightarrow\R$, $i\in\left\{1,2\right\}$, by
\begin{align*}
&h_1(x)=(g\!-\!1)\cos2a(x)+\cos2(g\!-\!1)a(x),\\
&h_2(x)=-(g\!-\!1)\sin2a(x)+\sin2(g\!-\!1)a(x).
\end{align*}
Clearly,
$$W(x)=\tfrac{1}{2}r'(x)^2-\tfrac{m}{2g}h_1(x)\sin^2r(x)-\tfrac{m}{2g}h_2(x)\sin^2\big(r(x)\!-\!\tfrac{3\pi}{4}\big)$$
and
$$r''(x)=(m-1)\tanh x\,r'(x)+\tfrac{m}{2g}[\sin(2r(x))h_1(x)+\cos(2r(x))h_2(x)].$$

For later use we state some elementary properties of $h_1$ and $h_2$.

\begin{Lem}[Properties of $h_1$]
\label{h1}
\begin{enumerate}
\item $g=2$:
The function $h_1(\,\cdot\,)$ decreases strictly, is positive on the negative $x$-axis, zero for $x=0$ and negative on the positive $x$-axis.
\item $g=3$:
The function $h_1(\,\cdot\,)$ decreases strictly.
There exists a unique $z_0>0$ such that $h_1(z_0)=0$, $h_1(x)$ is positive for $x<z_0$ and negative for $x>z_0$.
\item $g=4$:
The function $h_1(\,\cdot\,)$ is positive and decreases strictly.
\item $g=6$:
The function $h_1(\,\cdot\,)$ is positive.
There exists a unique $x_0>0$ such that  $h_1'(x_0)=0$, $h_1'(x)$ is negative for $x<x_0$ and positive for $x>x_0$.
\end{enumerate}
\end{Lem}

\begin{Lem}[Properties of $h_2$]
\label{h2}
The function $h_2(\,\cdot\,)$ vanishes identically. For $g\in\left\{3,4,6\right\}$ we have $h_2(\,\cdot\,)<0$ and $h_2(\,\cdot\,)$ decreases strictly.
\end{Lem}

\subsubsection{The symmetric boundary value problem}
This subsection addresses the so-called symmetric $(g,m,\ell_1,\ell_2)$-boundary value problem.

\smallskip

The following lemma asserts that every solution of the $(g,m)$-differential equation with $r(0)=y_j:=(jg\!+\!1)\tfrac{\pi}{2g}$, $j\in\Z$, exhibits point symmetry with respect to the point $P_j=(0,y_j)$. The proof is straightforward and is thus omitted.

\begin{Lem}
\label{sym}
\label{reflect}
 Let $r$ be a solution of the $(g,m)$-differential equation.
\begin{enumerate}
\item The maps $x\mapsto 2y_j-r(-x)$, $j\in\Z$ solve the $(g,m)$-differential equation.
\item If $r(0)=y_j$, $j\in\Z$, then $r$ is point symmetric to $P_j$.
\end{enumerate}
\end{Lem}

Thus, every solution $r$ of the $(g,m)$-differential equation that meets the conditions $r(0)=y_j$ for some $j\in\Z$ and $\lim_{x\rightarrow\infty}r(x)=2y_k$ for some $k\in\Z$ must be point symmetric to $P_j$.
This boundary value problem will be referred to henceforth as the \textit{symmetric $(g,m,jg+1,kg+1)$-boundary value problem}.

\smallskip

Every symmetric $(g,m,j^{'},j^{''})$-boundary value problem solution produces a solution for a $(g,m,k)$-boundary value problem.
In fact, if $r$ solves the symmetric $(g,m,j^{'},j^{''})$-boundary value problem, then $r+(j^{''}-j^{'})\pi$ is a solution of the $(g,m,(2j^{''}-j^{'})g+1)$-boundary value problem.
The converse, however, does not have to hold true: not every solution of a \((g,m,k)\)-boundary value problem produces a solution of a symmetric \((g,m,j^{'},j^{''})\)-boundary value problem.
The reason for this is that there may be solutions to the $(g,m,k)$-boundary value problem that are not symmetric with respect to points like $P_i$. 
Our numerical studies show that this is the case, for instance, when $(g,m)=(3,2)$.

\begin{Bem}
   \begin{enumerate}
       \item Note that it suffices to examine $j\in\Z_2$, as there is clearly a one-to-one correspondence between the solutions of the $(g,m,j,k)$ boundary value problem and those of the $(g,m,j+2,k+2)$ boundary value problem. 
\item The linear solutions \( r(x)=a(x) \) and \( r(x)=-(g-1)a(x)+\pi \) of the boundary value problems \((g,m,1)\) and \((g,m,-(g-1))\), respectively, solve the symmetric boundary value problems \((g,m,0,0)\) and \((g,m,-1,-1)\).  
\item It can be easily demonstrated that the symmetric $(g,m,j,k)$-boundary value problems possess a linear solution if and only if $k=j$ and $j\in\{0,1\}$. 
   \end{enumerate} 
\end{Bem}

\subsection{Numerical investigations for the cases with $m=1$}
In this subsection, we present the outcomes of numerical investigations.

\smallskip

To generate solutions of the $(g,1,k)$-boundary value problem with $g\in\left\{2,3,4,6\right\}$, we utilize a shooting method at $t=0$.
The results of our numerical studies suggest that there is precisely one solution to the $(g,1,k)$-boundary value problem for each $k\in\{1-2g,1-g,1,1+g\}$.
For $k=1$ and $k=1-g$, these correspond to the known linear solutions $r(x)=a(x)$ and $r(x)=-(g-1)a(x)+\pi$, respectively.
For $k=1-2g$ and $k=1+g$, the numerically found solutions exhibit point symmetry with respect to $P_{-2}$ and $P_1$, respectively.
Moreover, for other selections of $k\in\Z$, it appears that no solutions to the $(g,1,k)$-boundary value problem exist.

\smallskip

The following diagram visualizes the results mentioned in the preceding paragraph for the case of $g=4$ and $m=1$. The graphs in pink and green represent $k=1$ and $k=1-g$, respectively.
The graphs in yellow and blue correspond to $k=1+g$ and $k=1-2g$, respectively.

\begin{center}
\includegraphics[width=8cm]{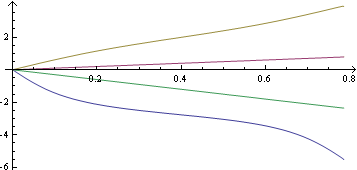}
\end{center}

\smallskip

In light of the numerical investigations mentioned above, we will restrict ourselves to the symmetric $(g,1,k^{'},k^{''})$-boundary value problems.
To put it differently, we look at solutions $r$ of the $(g,1)$-differential equation with initial conditions $r(0)=\tfrac{(jg\!+\!1)\pi}{2g}$ and $r'(0)=v$, and we seek initial velocities that make $r$ a solution of a symmetric $(g,1,jg+1,k)$-boundary value problem.
When the dependence of $W$ on $v$ is essential for our considerations, we will refer to $W$ as $W_v$.

\smallskip

Numerical investigations suggest that the sign of $W_v(-\infty)$ gives some insight into the behavior of $r$: 
\begin{itemize} \item if $W_{v}(-\infty)=0$, then $r$ is a solution to a symmetric $(g,1,jg+1,k)$-boundary value problem; \item if $W_{v}(-\infty)>0$, then $r$ is either strictly increasing or strictly decreasing and satisfies $\lim_{x\rightarrow\infty}r(x)=\pm\infty$; \item if $W_{v}(-\infty)<0$ and $r(0)=(g+1)\tfrac{\pi}{2g}$, then for all $x\in\R$, it holds that $0\leq r(x)\leq (g+1)\tfrac{\pi}{g}$; \item if $W_{v}(-\infty)<0$ and $r(0)=(1-2g)\tfrac{\pi}{2g}$, then for all $x\in\R$, the inequality $(1-2g)\tfrac{\pi}{g}\leq r(x)\leq 0$ holds, and in both previous cases, $r'$ changes sign infinitely often. 
\end{itemize}

For each of the choices $m=1$, $g\in\left\{2,3,4,6\right\}$, the numerical results indicate that there exist 
$l_{j},u_{j}\in\R$ such that $W_{v}(-\infty)<0$ if $v\in\left(l_j,u_j\right)$, $W_{v}(-\infty)=0$ if $v\in\left\{l_j,u_j\right\}$
and $W_{v}(-\infty)>0$ for the remaining cases. The numerical investigations yield 
\begin{center}
\begin{tabular}{c||c|c|c|c}
\hline
$g$&2&3&4&6\\ \hline
$u_1$&0.881&0.954&0.975&0.989\\ \hline
$l_{0}$&-0.881&-0.750&-0.648&-0.507\\ \hline
\end{tabular}
\end{center}
Furthermore, they indicate $l_1=-\frac{g-1}{g}$ and $u_{0}=\frac{1}{g}$, respectively, which are the initial velocities of the linear solutions.

\subsection{Proof of existence}
For each $g\in\left\{2,3,4,6\right\}$ we prove
the existence of two solutions of the $(g,1)$-BVP which do not coincide with linear solutions.
\begin{Lem}
\label{gw}
If $\lim_{x\rightarrow\infty}r(x)$ exists and is finite, then $\lim_{x\rightarrow\infty}r(x)=k\frac{\pi}{2}+\frac{\pi}{g}$, $k\in\Z$. 
If $\lim_{x\rightarrow -\infty}r(x)$ exists and is finite, then $\lim_{x\rightarrow -\infty}r(x)=k\tfrac{\pi}{2}$, $k\in\Z$.
\end{Lem}
\begin{proof}
Assume that $\lim_{x\rightarrow\infty}r(x)=c$ for a $c\in\R$.
When $x$ goes to infinity, the right hand side of the $(g,1)$-differential equation converges to
\begin{align*}
\tfrac{1}{2g}\left[(g\!-\!1)\sin\big(2c\!-\!2\tfrac{\pi}{g}\big)+\sin\big(2c\!+\!2(g\!-\!1)\tfrac{\pi}{g}\big)\right].
\end{align*}
This term must disappear, or else $\lim_{x\rightarrow\infty}r(x)$ would be nonexistent.
This leads to the conclusion. The second case is handled similarly.
\end{proof}
 
\begin{Lem}
\label{wincrease}
If $m=1$ and
\begin{enumerate}
	\item $g\in\left\{2,3,4\right\}$, then the function $W$ increases strictly.
	\item and $g=6$, then the function $W$ increases strictly on the negative $x$-axis. 
\end{enumerate}
\end{Lem}
\begin{proof}
Use Lemma\,\ref{h1} and Lemma\,\ref{h2}.
\end{proof}

\begin{Lem}
\label{wnegativ}
For $m=1$ we have $\lim_{x\rightarrow\pm\infty}W'(x)=0$. 
Furthermore, the limits $W(\pm\infty):=\lim_{x\rightarrow \pm\infty}W(x)$ exist and are finite.
\end{Lem}
\begin{proof}
The first statement follows immediately from Subsection\,\ref{wr}.

\smallskip

Since $m=1$, we get $W'(x)=-\tfrac{1}{2g}\big(h_1'(x)\sin^2r(x)+h_2'(x)\sin^2(r(x)-\tfrac{3\pi}{4})\big)$,
with $h'_i(x):=\tfrac{d}{dx}h_i(x)$, $i\in\left\{1,2\right\}$.
Let $g\in\left\{2,3,4\right\}$ and $a\in\R$ be given.
Then the above identity implies 
\begin{align*}
W(x)=W(a)+\int_{a}^x-\tfrac{1}{2g}\big(h_1'(\xi)\sin^2r(\xi)+h_2'(\xi)\sin^2(r(\xi)-\tfrac{3\pi}{4})\big)d\xi.
\end{align*}
Lemma\,\ref{h1} and Lemma\,\ref{h2} imply that the integrand is positive.
Hence, the function defined by $x\mapsto W_v(x)$
is increasing. Furthermore, it is bounded from above, since
\begin{align*}
W(x)-W(a)=&\int_{a}^x-\tfrac{1}{2g}\big(h_1'(\xi)\sin^2r(\xi)+h_2'(\xi)\sin^2(r(\xi)-\tfrac{3\pi}{4})\big)d\xi\\
\leq &\int_{a}^x-\tfrac{1}{2g}\big(h_1'(\xi)+h_2'(\xi)\big)d\xi\\
=&-\tfrac{1}{2g}\big(h_1(x)+h_2(x)-h_1(a)-h_2(a)\big).
\end{align*}
Since the right hand side of this inequality is bounded by a constant, the limit $\lim_{x\rightarrow\infty}W(x)$ exists and is finite. 
Similarly, we prove that $\lim_{x\rightarrow -\infty}W(x)$ exists and is finite.

\smallskip

Below let $g=6$ and $a\in\R$ be given. Introduce $I_j:\R\rightarrow\R$, $j\in\{1,2\}$, by
\begin{align*}
&I_1:\R\rightarrow\R,\, &&x\mapsto \int_a^x-\tfrac{1}{12}h_2'(\xi)\sin^2(r(\xi)-\tfrac{3\pi}{4})d\xi,\\
&I_2:\R\rightarrow\R,\, &&x\mapsto \int_a^x-\tfrac{1}{12}h_1'(\xi)\sin^2r(\xi)d\xi.
\end{align*}
Clearly, $W(x)=W(a)+I_1(x)+I_2(x)$. As above we prove that  $I_1$ is increasing and bounded from above. 
Therefore the limit $\lim_{x\rightarrow\infty}I_1(x)$ exists and is finite.\\
By Lemma\,\ref{h1} there exists a unique zero $x_0>0$ of the function $h_1'(x)$ such that $h_1'(x)$ is negative for $x<x_0$ and positive for $x>x_0$.
Let $a>x_0$ be given. The function $I_2$
 is decreasing on $\left[a,\infty\right)$.
Since it is also bounded from below, the limit $\lim_{x\rightarrow\infty}I_2(x)$ exists and is finite.
Consequently, $\lim_{x\rightarrow\infty}W(x)=\lim_{x\rightarrow\infty}(I_1(x)+I_2(x))$ exists and is finite. Similarly, we prove that
$\lim_{x\rightarrow -\infty}W(x)$ exists and is finite.
\end{proof}

Motivated by the previous lemma, we define the function $$L:\R\rightarrow\R,\, v\mapsto W_v(-\infty).$$

\begin{Lem}
\label{wneg}
Let $r_v$ be a solution of the symmetric $(g,1,j)$-differential equation.
If $L(v)<0$ then $\lim_{x\rightarrow -\infty}r_v(x)=k\pi$ for a $k\in\Z$ is not possible. Furthermore, the function
$r_v$ is bounded.
\end{Lem}

\begin{proof}
If $\lim_{x\rightarrow -\infty}r_v(x)=k\pi$ we get $W_v(-\infty)\geq 0$,
which contradicts our assumption.
For any point $x_0\in\R$ with $r_v(x_0)=k\pi$, $k\in\Z$, we get $W_v(x_0)\geq 0$. Since $W_v(-\infty)<0$, there thus has to exist a point $x_1\in\R$, such that $r_v(x)\neq k\pi$ for all $x\leq x_1$ and all $k\in\Z$.
Hence, since $r_v$ is continuous and point symmetric with respect to $P_j$, implies that the function $r_v$ is bounded.
\end{proof}

\begin{Lem}
\label{wpositiv}
Let $r_v$ be a solution of the symmetric $(g,1,j)$-differential equation and suppose $L(v)>0$. Then there exist a point $x_0\in\R$ such that $r_v'(x)\neq 0$ for all $x\leq x_0$. Furthermore, $\lim_{x\rightarrow -\infty}r_v'(x)=0$ is not possible.
\end{Lem}

\begin{proof}
Introduce the short hand notation $w:=W_v(-\infty)$.
By assumption there exists a point $x_1\in\R$ such that $W_v(x)>\tfrac{w}{2}$ for all $x\leq x_1$.
Furthermore, since $\lim_{x\rightarrow -\infty}h_1(x)=g$
and $\lim_{x\rightarrow -\infty}h_2(x)=0$, there exists a $x_2\in\R$ such that $W_v(x)-\tfrac{1}{2}r_v'(x)^2<\tfrac{w}{2}$ for all $x\leq x_2$.
Define $x_0=\mbox{min}\left(x_1,x_2\right)$. 
Hence,
\begin{align*}
\tfrac{w}{2}-\tfrac{1}{2}r_v'(x)^2<W_v(x)-\tfrac{1}{2}r_v'(x)^2<\tfrac{w}{2}
\end{align*}
for all $x\leq x_0$ and thus $r_v'(x)\neq 0$ for all $x\leq x_0$.\\
If $\lim_{x\rightarrow -\infty}r'(x)=0$, there exists a point $x_3\in\R$ such that $\tfrac{1}{2}r'(x)^2<\tfrac{w}{4}$
for all $x\leq x_3$. Define $x_4=\mbox{min}\left(x_2,x_3\right)$. Then $W_v(x)<\tfrac{3w}{4}$ for all $x\leq x_4$.
This is a contradiction, since by assumption there has to exist a point $x_5\in\R$ such that $W_v(x)>\tfrac{3w}{4}$ for all $x\leq x_5$.
\end{proof}

\begin{Cor}
\label{cip}
Let $r_v$ be a solution of the symmetric $(g,1,j)$-differential equation and suppose $L(v)>0$.
Then we have $\lim_{x\rightarrow -\infty}r(x)=\pm\infty$ and $\lim_{x\rightarrow \infty}r(x)=\pm\infty$.
\end{Cor}

\begin{Lem}
\label{wneg1}
Let $g\in\left\{2,3,4,6\right\}$.
Suppose that $r$ is point symmetric with respect to the point $P_1^x=(0,(g\!+\!1)\tfrac{\pi}{2g})$. Then
there exists $v_{+}>0$ with $r'(0)=v_{+}$ such that $L(v_+)=0$.
Furthermore, there exists $v_{-}<0$ with $r'(0)=v_{-}$ such that $L(v_-)=0$.
\end{Lem}
\begin{proof}
First we show that $L(0)<0$. Afterwards we verify that $L(v)>0$ if we chose $v\in\R$ with large enough absolute value.
Hence, the intermediate value theorem implies the existence of initial velocities $v_{+}>0$ and $v_{-}<0$
with $L(v_+)=L(v_-)=0$.

\smallskip

For $g=2$ we get $W_v(0)=\tfrac{1}{2}r'(0)^2=\tfrac{1}{2}v^2$. Hence, if we chose $v=0$ we get $W_0(0)=0$.
Consequently, Lemma\,\ref{wincrease} implies $W_0(-\infty)<0$. 
For $g\in\left\{3,4,6\right\}$ we get $W_v(0)=\tfrac{1}{2}r'(0)^2-c_g$, where $c_g$ is a positive constant.
Hence, if we chose $r'(0)=0$ we get $W_0(0)<0$.
Consequently, Lemma\,\ref{wincrease} implies $W_0(-\infty)<0$. By Subsection\,\ref{wr}, we obtain
$\lvert W'_v(x)\rvert \leq 4\tfrac{g-1}{g^2\cosh x}$
for all $x\in\R$ and thus
\begin{align*}
-2\pi\tfrac{g-1}{g^2}\leq W_v(0)-W_v(-\infty)\leq 2\pi\tfrac{g-1}{g^2}.
\end{align*}
Consequently, if we chose $v>0$ large enough we get $L(v)=W_v(-\infty)>0$.
Similarly, if we chose $v<0$ with large enough absolute value we get $L(v)=W_v(-\infty)>0$.

\smallskip

Since $L(v)$ depends continuously on $v$, the intermediate value theorem implies
the existence of a $v_{+}>0$ with $r'(0)=v_{+}$ such that $L(v_+)=0$.
Furthermore, there exists $v_{-}<0$ with $r'(0)=v_{-}$ such that $L(v_-)=0$.
\end{proof}

The numerical outcomes show that the assertion of the previous lemma holds true also under the assumption that $r$ is point symmetric with respect to $P_{-2}^x=(0,(1\!-\!2g)\tfrac{\pi}{2g})$. However, the previous proof cannot be applied in this case, as $W_0(0)>0$ for $g\in\left\{3,4,6\right\}$.

\begin{Lem}
\label{wneg2}
Let $g\in\left\{2,3,4,6\right\}$ and $r$ be point symmetric with respect to the point $P_{-2}^x=(0,(1\!-\!2g)\tfrac{\pi}{2g})$. Then
there exist $v_{\pm}\in\R$ with $r'(0)=v_{\pm}$ and $v_{-}<v_{+}$ such that $L(v_{\pm})=0$.
\end{Lem}
\begin{proof}
The proof concept is identical to that of the proof of Lemma\,\ref{wneg1}.
For $g=2$, the proof from Lemma\,\ref{wneg1} applies, leaving us with the cases where $g\in\left\{3,4,6\right\}$.

\smallskip

We treat the three possible cases for $g$ separately. 
We show the existence of $v_0\in\R$ such that $W_{v_0}(-\infty)<0$ as follows: for every $v\in\R$, we establish that $W_v(\infty)-W_v(0)<c_1$, where $c_1$ is a constant in $\R$. We also establish that $c_2:=W_v(\infty)-W_v(-\infty)$, which does not depend on $v$.
Subsequently, we demonstrate that it is possible to find a value $v_0\in\R$ for which the following holds: $W_{v_0}(\infty)<c_1+W_{v_0}(0)<c_2$. Therefore, $W_{v_0}(-\infty)<0$, which is what we wanted to show.
The proof of the lemma is then completed by proceeding as in the proof of Lemma\,\ref{wneg1}.

\smallskip

To begin with, let us assume $g=3$ and $r'(0)=-\epsilon$, where $\epsilon>0$.
Then, there is a point $x_0\in\R$ for which $r(x_0)=-\pi$, and for every $x\in\left[0,x_0\right]$, it holds that $-\pi\leq r(x)\leq-\tfrac{5\pi}{6}$.
This can be easily derived from $r(0)=-\tfrac{5\pi}{6}$ and $r'(0)<0$, along with the $(g,1)$-ODE.
Moreover, the $(g,1)$-ODE entails that for every $x\in\R$, $r''(x)$ lies within the bounds of $-\tfrac{1}{2}$ and $\tfrac{1}{2}$. Therefore, for every $x\in\R$, it holds that $-\tfrac{1}{2}x\leq r'(x)-r'(0)\leq\tfrac{1}{2}x$.
Hence, we obtain
$-1\leq r'(x)\leq 1$ for all $x$ with $0\leq x\leq 2-2\epsilon$.
Thus we get $x_0\geq\frac{\pi}{6}$. Below let $x_0=\frac{\pi}{6}$.
Since $\lvert r'(x)\lvert\leq 1$ for $0\leq x\leq 2-2\epsilon$ we get
$-\frac{7\pi}{6}\leq r(x)\leq -\frac{5\pi}{6}$ for all $x\in\left[0,\frac{\pi}{3}\right]$.
Using the identity for $W_v'(x)$ derived in Subsection\,\ref{wr} we get
\begin{align*}
\int_0^xW_v'(\xi)d\xi\leq\,&\tfrac{1}{4}\int_0^{\tfrac{\pi}{3}}\tfrac{2}{9\cosh\xi}\big(\sin2a(\xi)+\sin4a(\xi)\big)d\xi\\
+&\,\tfrac{1}{8}(1+\sqrt{3})^2\int_0^{\tfrac{\pi}{3}}-\tfrac{2}{9\cosh\xi}\big(-\cos2a(\xi)+\cos4a(\xi)\big)d\xi\\
+&\,\int_{\frac{\pi}{3}}^{\infty}\tfrac{2}{9\cosh\xi}\big(\sin2a(\xi)+\sin4a(\xi)\big)d\xi\\
+&\,\int_{\frac{\pi}{3}}^{\infty}-\tfrac{2}{9\cosh\xi}\big(-\cos2a(\xi)+\cos4a(\xi)\big)d\xi
\end{align*}
where we make use of $\sin^2(r)\leq\tfrac{1}{4}$ and $\sin^2(r-\tfrac{3\pi}{4})\leq\tfrac{1}{8}(1+\sqrt{3})^2$ as long as
$-\tfrac{7\pi}{6}\leq r\leq -\tfrac{5\pi}{6}$.
Thus, $W_{-\epsilon}(x)\leq W_{-\epsilon}(0)+0.41\leq\tfrac{1}{2}\epsilon^2+0.53$ for all $x\geq 0$. Hence,
$W_{-\epsilon}(\infty)\leq \frac{1}{2}\epsilon^2+0.53.$
Furthermore, one proves $W_v(\infty)=W_v(-\infty)+\tfrac{1}{8}(3+\sqrt{3})$ for all $v\in\R$.
Consequently, $W_{-\epsilon}(-\infty)\leq\tfrac{1}{2}\epsilon^2+0.53-\tfrac{1}{8}(3+\sqrt{3})$.
Thus, for a sufficiently small value of $\epsilon>0$, we have $W_{-\epsilon}(-\infty)<0$.

\smallskip

We will assume $g=4$ and $r'(0)=-\epsilon$, where $\epsilon>0$.
Then there is a point $x_0\in\R$ such that $r(x_0)=-\frac{5\pi}{4}$ and for all $x\in\left[0,x_0\right]$, $-\frac{5\pi}{4}\leq r(x)\leq-\frac{7\pi}{8}$.
This can be easily deduced from $r(0)=-\frac{7\pi}{8}$ and $r'(0)<0$, along with the $(g,1)$-ODE.
Moreover, the $(g,1)$-ODE leads to the conclusion that $-\tfrac{1}{2}\leq r''(x)\leq\tfrac{1}{2}$ for every $x\in\R$. Therefore, for every $x\in\R$, it holds that $-\tfrac{1}{2}x\leq r'(x)-r'(0)\leq\tfrac{1}{2}x$. Thus, for every $x$ such that $0\leq x\leq 2-2\epsilon$, it holds that $-1\leq r'(x)\leq 1$.
So we obtain $x_0\geq\tfrac{3\pi}{8}$. In the following, we assume that $x_0=\tfrac{3\pi}{8}$.
Therefore, it holds that $-\tfrac{5\pi}{4}\leq r(x)\leq -\tfrac{7\pi}{8}$ for every $x\in\left[0,\tfrac{3\pi}{8}\right]$.
Using the identity for $W_v'(x)$ derived in Subsection\,\ref{wr}, we get
\begin{align*}
\int_0^xW_v'(\xi)d\xi\leq\,&\sin^2\tfrac{7\pi}{8}\int_0^{\tfrac{\pi}{4}}\tfrac{3}{16\cosh\xi}\big(\sin2a(\xi)+\sin4a(\xi)\big)d\xi\\
+&\,\tfrac{1}{2}\int_{\tfrac{\pi}{4}}^{\tfrac{3\pi}{8}}\tfrac{3}{16\cosh\xi}\big(\sin2a(\xi)+\sin4a(\xi)\big)d\xi\\
+&\,\int_{\tfrac{3\pi}{8}}^{\infty}\tfrac{3}{16\cosh\xi}\big(\sin2a(\xi)+\sin4a(\xi)\big)d\xi\\
+&\,\int_0^{\infty}-\tfrac{3}{16\cosh\xi}\big(-\cos2a(\xi)+\cos4a(\xi)\big)d\xi,
\end{align*}
where we make use of $\sin^2r\leq\sin^2(\tfrac{7\pi}{8})$ for $-\tfrac{9\pi}{8}\leq r\leq -\tfrac{7\pi}{8}$ and $\sin^2r\leq\tfrac{1}{2}$ for
$-\tfrac{5\pi}{4}\leq r\leq -\tfrac{9\pi}{8}$. 
Consequently, we specifically take $r(\tfrac{\pi}{4})=-\tfrac{9\pi}{8}$ as given. In all other possible scenarios, the aforementioned inequality improves further.
Therefore, for all $x\geq 0$, it holds that $W_{-\epsilon}(x)\leq W_{-\epsilon}(0)+0.34\leq \tfrac{1}{2}\epsilon^2+0.47$.
Thus, $W_{-\epsilon}(\infty)\leq \tfrac{1}{2}\epsilon^2+0.47$.
Given that $W_v(\infty)=W_v(-\infty)+\tfrac{1}{2}$ for every $v\in\R$, we find that $W_{-\epsilon}(-\infty)\leq \tfrac{1}{2}\epsilon^2+0.47-\tfrac{1}{2}$. Therefore, if $\epsilon>0$ is sufficiently small, it follows that $W_{-\epsilon}(-\infty)<0$.

\smallskip

We will take $g=6$ and $r'(0)=-\epsilon$ with $\epsilon>0$ in what follows.
Then there is a point $x_0\in\R$ such that $r(x_0)=-\tfrac{5\pi}{4}$ and for all $x\in\left[0,x_0\right]$, it holds that $-\tfrac{5\pi}{4}\leq r(x)\leq-\tfrac{11\pi}{12}$.
This can be easily derived from $r(0)=-\tfrac{11\pi}{12}$ and $r'(0)<0$, along with the $(6,m)$-ODE.
Moreover, this ODE entails that $-\tfrac{1}{2}\leq r''(x)\leq\tfrac{1}{2}$ for every $x\in\R$. Therefore, for every $x\in\R$, it holds that $-\tfrac{1}{2}x\leq r'(x)-r'(0)\leq\tfrac{1}{2}x$. Thus, for every $x$ such that $0\leq x\leq 2-2\epsilon$, it holds that $-1\leq r'(x)\leq 1$.
Hence, we obtain $x_0\geq\tfrac{\pi}{3}$. 
For our further considerations, we take $x_0=\tfrac{\pi}{3}$.
Hence, we have $-\tfrac{5\pi}{4}\leq r(x)\leq -\frac{11\pi}{12}$ for all $x\in\left[0,\tfrac{\pi}{3}\right]$.
Using the identity for $W_v'(x)$ derived in Subsection\,\ref{wr} we get
\begin{align*}
\int_0^xW_v'(\xi)d\xi\leq\,&\sin^2\tfrac{11\pi}{12}\int_0^{\tfrac{\pi}{6}}\tfrac{5}{36\cosh\xi}\big(\sin2a(\xi)+\sin4a(\xi)\big)d\xi\\
+&\,\sin^2\tfrac{5\pi}{4}\int_{\tfrac{\pi}{6}}^{z_1}\tfrac{5}{36\cosh\xi}\big(\sin2a(\xi)+\sin4a(\xi)\big)d\xi\\
+&\,\sin^2\tfrac{5\pi}{3}\int_0^{\tfrac{\pi}{6}}-\tfrac{5}{36\cosh\xi}\big(-\cos2a(\xi)+\cos4a(\xi)\big)d\xi\\
+&\,\tfrac{1}{4}\int_{\tfrac{\pi}{6}}^{\tfrac{\pi}{3}}-\tfrac{5}{36\cosh\xi}\big(-\cos2a(\xi)+\cos4a(\xi)\big)d\xi\\
+&\,\int_{\tfrac{\pi}{3}}^{\infty}-\tfrac{5}{36\cosh\xi}\big(-\cos2a(\xi)+\cos4a(\xi)\big)d\xi,
\end{align*}
where $z_1$ defines the unique zero of $h_1'(x,6)$. Furthermore, we make use of $\sin^2r\leq\sin^2(\tfrac{11\pi}{12})$ and $\sin^2(r-\tfrac{3\pi}{4})\leq\sin^2(\tfrac{5\pi}{3})$ for $-\tfrac{13\pi}{12}\leq r\leq -\tfrac{11\pi}{12}$;
$\sin^2r\leq \sin^2(\tfrac{5\pi}{4})$ and $\sin^2(r-\tfrac{3\pi}{4})\leq \tfrac{1}{4}$ for
$-\tfrac{5\pi}{4}\leq r\leq -\tfrac{13\pi}{12}$.
We in particular assume $r(\tfrac{\pi}{6})=-\tfrac{13\pi}{12}$. In all other possible scenarios, the aforementioned inequality improves further.
In consequence, for every $x\geq 0$ it holds that $ W_{-\epsilon}(x)\leq W_{-\epsilon}(0)+0.19\leq \tfrac{1}{2}\epsilon^2+0.3 $, which implies $W_{-\epsilon}(\infty)\leq \tfrac{1}{2}\epsilon^2+0.3$.
As $W_v(\infty)=W_v(-\infty)+\tfrac{1}{8}(1+\sqrt{3})$ is true for every $v\in\R$, this implies that $W_{-\epsilon}(-\infty)\leq \tfrac{1}{2}\epsilon^2+0.3-\tfrac{1}{8}(1+\sqrt{3})$.
Therefore, for sufficiently small $\epsilon>0$, it holds that $W_{-\epsilon}(-\infty)<0$.

\smallskip

For each case \( g \in \{ 3, 4, 6 \} \), we can now show the assertion using the same method as in Lemma\,\ref{wneg1}.
\end{proof}

\begin{Lem}
\label{lneg}
Let $g\in\left\{2,3,4,6\right\}$.
Suppose that $r$ is point symmetric with respect to $(0,(g+1)\tfrac{\pi}{2g})$. Let $v\in\R$ with $r'(0)=v$ be such that $L(v)<0$. Then there exists $k\in\Z$ and $x_0\in\mathbb{R}$ such that
$(k-1)\pi<r(x)<k\pi$ for all $x\leq x_0$.
\end{Lem}
\begin{proof}
By Lemma\,\ref{wnegativ} the limit $\lim_{x\rightarrow -\infty}W(x)=\lim_{x\rightarrow -\infty}(\frac{1}{2}r'(x)^2-\frac{m}{2}\sin^2r(x))$ exists.
By assumption this limit is negative, whence the claim.
\end{proof}

Having the preparation ready, we can now prove Theorem\,\ref{thm11}:

\begin{Thm}
\label{thm1}
For every $g\in\{2,3,4,6\}$, there are four (not necessarily distinct) values of $k\in\mathbb{Z}$ such that the $(g,1,k)$-boundary value problem has a solution.
\end{Thm}
\begin{proof}
For each possible choice of $g$, Lemmas \,\ref{wneg1} and \ref{wneg2} imply the existence of two real numbers $v_{\pm}\in\mathbb{R}$ with $v_-<v_+$ such that $L(v_{\pm})=0$.
Further, from the proofs of Lemmas \,\ref{wneg1} and \ref{wneg2} we have that there exist $v\in(v_-,v_+)$ such that $L(v)<0$.
From Lemma\,\ref{lneg}, the continuous dependence on $v$ and $\lim_{x\rightarrow -\infty}W_{v_{\pm}}(x)=\lim_{x\rightarrow -\infty}(\frac{1}{2}r'(x)^2-\frac{m}{2}\sin^2r(x))=0$ we obtain
the claim.
\end{proof}

\begin{Bem}
\label{bem1}
\begin{enumerate}
\item As mentioned above, the numerics indicate that the four values of $k\in\mathbb{Z}$ are different and are given by $1-2g,1-g,1,1+g$, but we did not prove this.
    \item Theorem\,\ref{thm1} establishes, by construction, the existence of four harmonic self-maps of $\s^n$ for $n\in\{3,4,5,7\}$. It should be noted that, for given $g\in\{2,3,4,6\}$, there are only two known solutions to the corresponding boundary problem. As a result, at least two of the constructed solutions are new (and, as previously stated, the numerics suggest that exactly two solutions are new). Note further that these solutions were not constructed in \cite{MR1436833, MR3427685}, as both papers only construct solutions with $m_0>1$.
\item 
Like the introduction already stated, it is known that harmonic self-maps of $\s^n$ exist for $n\in\{3,4,5,7\}$; see e.g. \cite{MR1436833, MR3427685}.
The construction method described in the aforementioned manuscripts, however, makes decisive use of the fact that for $g=1, 2\leq m\leq 5$ and $g=2, 2\leq m_0\leq 5, m_0\leq m_1$, the $(g,m_0,m_1)$-ODE is derived from an oscillation equation. This allows for the construction of an infinite number of solutions in these cases. Other constructions of harmonic maps, like R. Wood's construction of so-called eiconals found on pages 80 and 102 of \cite{MR2044031}, take advantage of special geometric situations.
To our knowledge, there are no results in the literature where the methods employed here are utilized to construct harmonic maps.
\end{enumerate}
\end{Bem}

Theorem\,\ref{thm1} also implies the existence of specific maps between special orthogonal groups.
The following is taken from \cite{MR4000241}:
Any isometric cohomogeneity one action $G\times \s^{n+1}\to \s^{n+1}$ can be lifted to an isometric cohomogeneity one action of $G\times\SO(n+1)$ on $\SO(n+2)$ with the metric $\tfrac{1}{2} \tr X\tp Y$ via \begin{gather*} G \times \SO(n+1) \times \SO(n+2) \to \SO(n+2), \quad \left( A,\bmat 1 & \\ & B\emat \right) \cdot C = AC\bmat 1 & \\ & B^{-1}\emat. \end{gather*}
The lift of a $(g,m_0,m_1)$-action on $\s^{n+1}$ is referred to as a {\em $(g,m_0,m_1)$-action on $\SO(n+2)$} (or simply a $(g,m)$-action if $m_0 = m_1$). 

For the $(g,m_0,m_1)$-action on $\SO(n+2)$ the following two theorems were obtained in \cite{MR4000241}:

\begin{Thm}[Theorem\,F of \cite{MR4000241}]
Given a $(g,m_0,m_1)$-action on $\SO(n+2)$ with $n = \frac{m_0+m_1}{2}g$, the normal component of the tension field of the map $g\cdot \tilde\gamma(2t)\to g\cdot\tilde\gamma(2r(t))$ vanishes if and only if $r$ solves the $(2g,m_0,m_1,k)$-boundary value problem.
\end{Thm}

\begin{Thm}[see Theorem\,H in \cite{MR4000241}]
Given a $(g,m_0,m_1)$-action on $\SO(n+2)$ the tangential component of the tension field of any $(k,r)$-map vanishes except possibly for
\begin{gather*}
  (g,m_0,m_1) \in \{ (4,2,2\ell+1), (4,4,4\ell+3), (4,4,5), (4,6,9) \}.
\end{gather*}
\end{Thm}

Thus from Theorem\,\ref{thm1}, we obtain the following theorem:
\begin{Thm}
\label{so}
For $\ell\in\{5,6\}$, there are four harmonic self-maps of $\SO(\ell)$.
\end{Thm}

\begin{Bem}
It is worth mentioning that the existence of (non-trivial) harmonic self-maps of $\SO(5), \SO(6)$ has been previously established, this follows e.g. from \cite{MR1436833} and the previously mentioned findings related to special orthogonal groups from \cite{MR4000241}.
However, for each $\ell\in\{5,6\}$, two of the harmonic self-maps of $\SO(\ell)$ are new, to the best of the authors' knowledge.
\end{Bem}

\section{Harmonic maps between non-compact cohomogeneity one manifolds}
\label{sec-nc}
We provide the existence of equivariant harmonic maps between non-compact cohomogeneity one manifolds.

\smallskip

We study non-compact cohomogeneity one manifolds of the form
\begin{align}
\label{mfd-nc}
 M=\mathbb{R}^{k_0+1}\times G_1/H_1\times\dots\times G_m/H_m,   
\end{align}
where $k_0\in\mathbb{N}$ with $k_0\geq 2$ and each $G_i/H_i$ is a compact, irreducible homogeneous space of dimension $k_i$ that is not flat.
Let
\begin{align*}
    G=\SO(k_0+1)\times G_1\times\dots\times G_m.
\end{align*}
The principal orbit is given by
\begin{align*}
 \s^{k_0}\times G_1/H_1\times\dots\times G_m/H_m,   
\end{align*}
the singular orbit is given by
\begin{align*}
 N=\{0\}\times G_1/H_1\times\dots\times G_m/H_m.   
\end{align*}
We assume that $M/G=[0,\infty)$.
For each $i\in\{1,\dots,m\}$, let $g_i$ be a $G_i$-invariant metric on $G_i/H_i$.
Each $G$-invariant smooth metric $g$ on $M\setminus{N}$ is of the form
\begin{align}
\label{metric}
 g=dt^2+f_0^2(t)g_{\s^{k_0}}+\sum_{j=1}^mf_j^2(t)g_i,  
\end{align}
where $f_0,f_1,\dots,f_m:(0,\infty)\rightarrow\mathbb{R}_+$ are smooth functions.
In order for the metric $g$ to extend smoothly to $N$, by \cite{MR1758585,MR4400726} we have to impose the following initial values
\begin{align*}
 f_0(0)=\dot f_i(0)=0, \dot f_0(0)=1,f_i(0)=a_i,     
\end{align*}
for $i\in\{1,\dots,m\}$,
where $a_i>0$.

\smallskip

Below let $\gamma:\mathbb{R}\rightarrow M$ be a fixed complete, unit speed normal geodesic of $M$.
We study equivariant maps of the form
\begin{align*}
  \psi:M\rightarrow M, g\cdot\gamma(t)\mapsto g\cdot\gamma(r(t)), 
\end{align*}
where $r:(0,\infty)\rightarrow\mathbb{R}$ is a smooth function with
$r(0)=0$.
Following Lemma 2.1 in \cite{MR2480860} one shows that these maps are smooth.

By (\ref{normalone}), (\ref{tanaltpart}) and (\ref{metric}) the construction of harmonic self-maps reduces (in the setting under consideration) to providing smooth solutions $r:[0,\infty)\rightarrow\mathbb{R}$ to the singular boundary value problem
\begin{align}
\label{bvpoa}
    \ddot r(t)+ (\sum_{i=0}^mk_i\frac{\dot f_i(t)}{f_i(t)})\dot r(t)-\sum_{i=0}^mk_i\frac{f_i(r(t))\dot f_i(r(t))}{f_i^2(t)}=0
\end{align}
with $r(0)=0$ and $\dot r(t)=v$, with $v\in\mathbb{R}$.
From (\cite{S26}) we obtain that, given $f_i, i\in\{0,\dots,m\}$, the function $r$ can be extended smoothly to $t=0$.  

\smallskip

The following considerations are inspired by the work \cite{MR987757} of Ratto.
Consider metrics of the form
\begin{align}
\label{metric-2}
 \tilde{g}=dt^2+e^{2\alpha_0(t)}f_0^2(t)g_{\s^{k_0}}+\sum_{j=1}^me^{2\alpha_j(t)}f_j^2(t)g_i, 
\end{align}
where $\alpha_i:[0,\infty)\rightarrow\mathbb{R}$ are smooth functions such that 
\begin{align}
\label{ali}
  \alpha_i(t)=0  
\end{align}
 for all $t\in[0,\epsilon], i\in\{0,\dots,m\}$ for some $\epsilon>0$ given. The condition on the functions $\alpha_i$ ensures that 
(\ref{metric-2}) coincides with (\ref{metric}) in a neighborhood of the singular orbit $N$.
Using Theorems 3.3 and 3.5 from (\cite{S26}), we obtain that the tension field of 
the map 
\begin{align}
\label{psi}
  \psi:(M,\tilde{g})\rightarrow (M,g),\, g\cdot\gamma(t)\mapsto g\cdot\gamma(r(t))
\end{align}
vanishes if and only if
\begin{align}
\label{bvp}
\ddot r(t)+ (\sum_{i=0}^mk_i\frac{\dot f_i(t)}{f_i(t)}+\sum_{i=0}^mk_i\dot\alpha_i(t))\dot r(t)-\sum_{i=0}^mk_i\frac{f_i(r(t))\dot f_i(r(t))}{f_i^2(t)}=0.
\end{align}
We will construct smooth solutions $r:[0,\infty)\rightarrow\mathbb{R}$ of (\ref{bvp}) with $r(0)=0$.
For any $v\in\mathbb{R}$ there exists a smooth solution $r:[0,\epsilon]\rightarrow\mathbb{R}$ of (\ref{bvp}) with $r(0)=0$ and $\dot r(0)=v$ which depends continuously on $v$, see \cite{S26}.
Below let $v\in\mathbb{R}\setminus\{0\}$ and $r$ be associated solution.
We can assume without loss of generality that $\dot r(\epsilon)\neq 0$. Indeed, if $\dot r(\epsilon)=0$, then the smoothness of $r$ and $\dot r(0)=v\neq 0$ imply that 
there exists a $\epsilon_0\in(0,\epsilon)$ such that $\dot r(\epsilon_0)\neq 0$. We can then adapt the assumption (\ref{ali}) accordingly.

\smallskip

Let $r$ be any smooth extention of $r:[0,\epsilon]\rightarrow\mathbb{R}$ such that $r_{\lvert [\epsilon,\infty)}$ is strictly monotone.
Thus, from (\ref{bvp}) we get
\begin{align*}
\sum_{i=0}^mk_i\dot\alpha_i(t)=-\frac{\ddot r(t)}{\dot r(t)}-\sum_{i=0}^mk_i\frac{\dot f_i(t)}{f_i(t)}-\sum_{i=0}^mk_i\frac{f_i(r(t))\dot f_i(r(t))}{f_i^2(t)\dot r(t)},
\end{align*}
and therefore we have
\begin{align}
\label{eq-a}
\sum_{i=0}^mk_i\alpha_i(t)=\int_{\epsilon}^t(-\frac{\ddot r(s)}{\dot r(s)}-\sum_{i=0}^mk_i\frac{\dot f_i(s)}{f_i(s)}-\sum_{i=0}^mk_i\frac{f_i(r(s))\dot f_i(r(s))}{f_i^2(s)\dot r(s)})ds.
\end{align}
Clearly, for each given smooth $r$ there exist infinitely many functions $\alpha_i$ which satisfy (\ref{ali}) and (\ref{eq-a}).
In particular, we have thus established the existence of a harmonic map
$(M,\tilde{g})\rightarrow (M,g)$. Note that we can choose $r$ 
such that the constructed harmonic map is not trivial, 
i.e. a map that is not the identity, or the negative identity map.
Further, note that the constructed harmonic map is onto since $r$ is chosen to be strictly monotone on the interval $[\epsilon,\infty)$.
In particular, the constructed harmonic map can not be a constant map.

\smallskip

We have thus established Theorem\,\ref{B}:

\begin{Thm}
\label{main2}
Let $(M,g)$ be a non-compact cohomogeneity one manifold as described in (\ref{mfd-nc}, \ref{metric}).
Let $v\in\mathbb{R}\setminus\{0\}$ be given. Then there exists an $\epsilon>0$ such that the smooth $r$ solution of (\ref{bvpoa}) is defined on $[0,\epsilon]$ and $\dot r(\epsilon)\neq 0$. For any strictly monotone, smooth extension of $r$ to $[\epsilon,\infty)$ there exist infinitely many metrics of the form (\ref{metric-2}) such that the map (\ref{psi}) is harmonic.
\end{Thm}

\begin{Bem}
It is worth mentioning that Theorem\,\ref{main2} can be readily extended to smooth equivariant maps between non-compact manifolds that may differ from one another; for the sake of simplicity in presentation, we however restricted ourselves to the case at hand.
\end{Bem}

\bibliographystyle{plain}
\bibliography{mybib}

@article {MR987757,
    AUTHOR = {Ratto, Andrea},
     TITLE = {Harmonic maps from deformed spheres to spheres},
   JOURNAL = {Amer. J. Math.},
  FJOURNAL = {American Journal of Mathematics},
    VOLUME = {111},
      YEAR = {1989},
    NUMBER = {2},
     PAGES = {225--238},
      ISSN = {0002-9327,1080-6377},
   MRCLASS = {58E20},
  MRNUMBER = {987757},
MRREVIEWER = {Donato\ Scolozzi},
       DOI = {10.2307/2374509},
       URL = {https://doi.org/10.2307/2374509},
}

@article {MR2406265,
    AUTHOR = {Grove, Karsten and Wilking, Burkhard and Ziller, Wolfgang},
     TITLE = {Positively curved cohomogeneity one manifolds and 3-{S}asakian
              geometry},
   JOURNAL = {J. Differential Geom.},
  FJOURNAL = {Journal of Differential Geometry},
    VOLUME = {78},
      YEAR = {2008},
    NUMBER = {1},
     PAGES = {33--111},
      ISSN = {0022-040X,1945-743X},
   MRCLASS = {53C21 (53C26 57S25)},
  MRNUMBER = {2406265},
MRREVIEWER = {Andrew\ Swann},
       URL = {http://projecteuclid.org/euclid.jdg/1197320603},
}

@incollection {MR334094,
    AUTHOR = {Takagi, Ryoichi and Takahashi, Tsunero},
     TITLE = {On the principal curvatures of homogeneous hypersurfaces in a
              sphere},
 BOOKTITLE = {Differential geometry (in honor of {K}entaro {Y}ano)},
     PAGES = {469--481},
 PUBLISHER = {Kinokuniya Book Store, Tokyo},
      YEAR = {1972},
   MRCLASS = {53C35},
  MRNUMBER = {334094},
MRREVIEWER = {\`E.\ Vinberg},
}

@article {MR583825,
    AUTHOR = {M\"unzner, Hans Friedrich},
     TITLE = {Isoparametrische {H}yperfl\"achen in {S}ph\"aren},
   JOURNAL = {Math. Ann.},
  FJOURNAL = {Mathematische Annalen},
    VOLUME = {251},
      YEAR = {1980},
    NUMBER = {1},
     PAGES = {57--71},
      ISSN = {0025-5831,1432-1807},
   MRCLASS = {53C40},
  MRNUMBER = {583825},
MRREVIEWER = {Th.\ Friedrich},
       DOI = {10.1007/BF01420281},
       URL = {https://doi.org/10.1007/BF01420281},
}

@article {MR298593,
    AUTHOR = {Hsiang, Wu-yi and Lawson, Jr., H. Blaine},
     TITLE = {Minimal submanifolds of low cohomogeneity},
   JOURNAL = {J. Differential Geometry},
  FJOURNAL = {Journal of Differential Geometry},
    VOLUME = {5},
      YEAR = {1971},
     PAGES = {1--38},
      ISSN = {0022-040X,1945-743X},
   MRCLASS = {53C40},
  MRNUMBER = {298593},
MRREVIEWER = {Bang-yen\ Chen},
       URL = {http://projecteuclid.org/euclid.jdg/1214429775},
}

@article {MR653945,
    AUTHOR = {Eells, J. and Lemaire, L.},
     TITLE = {Deformations of metrics and associated harmonic maps},
   JOURNAL = {Proc. Indian Acad. Sci. Math. Sci.},
  FJOURNAL = {Indian Academy of Sciences. Proceedings. Mathematical
              Sciences},
    VOLUME = {90},
      YEAR = {1981},
    NUMBER = {1},
     PAGES = {33--45},
      ISSN = {0253-4142,0973-7685},
   MRCLASS = {58E20},
  MRNUMBER = {653945},
       DOI = {10.1007/BF02867016},
       URL = {https://doi.org/10.1007/BF02867016},
}

@article {MR1122903,
    AUTHOR = {Eells, J. and Ferreira, M. J.},
     TITLE = {On representing homotopy classes by harmonic maps},
   JOURNAL = {Bull. London Math. Soc.},
  FJOURNAL = {The Bulletin of the London Mathematical Society},
    VOLUME = {23},
      YEAR = {1991},
    NUMBER = {2},
     PAGES = {160--162},
      ISSN = {0024-6093,1469-2120},
   MRCLASS = {58E20},
  MRNUMBER = {1122903},
MRREVIEWER = {A.\ R.\ Aithal},
       DOI = {10.1112/blms/23.2.160},
       URL = {https://doi.org/10.1112/blms/23.2.160},
}

@book {MR716320,
    AUTHOR = {Baird, Paul},
     TITLE = {Harmonic maps with symmetry, harmonic morphisms and
              deformations of metrics},
    SERIES = {Research Notes in Mathematics},
    VOLUME = {87},
 PUBLISHER = {Pitman (Advanced Publishing Program), Boston, MA},
      YEAR = {1983},
     PAGES = {v+181},
      ISBN = {0-273-08603-0},
   MRCLASS = {58E20 (53C20)},
  MRNUMBER = {716320},
MRREVIEWER = {Samuel\ I.\ Goldberg},
}

@misc{S26,
     title={The Initial Value Problem
for Harmonic maps of Cohomogeneity One manifolds}, 
      author={Anna Siffert},
      year={2026},
      eprint={2601.15022},
      archivePrefix={arXiv},
      primaryClass={math.DG},
      url={https://arxiv.org/abs/2601.15022}, 
}

@article {MR1214054,
    AUTHOR = {Urakawa, Hajime},
     TITLE = {Equivariant harmonic maps between compact {R}iemannian
              manifolds of cohomogeneity {$1$}},
   JOURNAL = {Michigan Math. J.},
  FJOURNAL = {Michigan Mathematical Journal},
    VOLUME = {40},
      YEAR = {1993},
    NUMBER = {1},
     PAGES = {27--51},
      ISSN = {0026-2285,1945-2365},
   MRCLASS = {58E20 (53C30)},
  MRNUMBER = {1214054},
MRREVIEWER = {Martin\ A.\ Guest},
       DOI = {10.1307/mmj/1029004673},
       URL = {https://doi.org/10.1307/mmj/1029004673},
}

@article {MR1436833,
    AUTHOR = {Bizo\'n, Piotr and Chmaj, Tadeusz},
     TITLE = {Harmonic maps between spheres},
   JOURNAL = {Proc. Roy. Soc. London Ser. A},
  FJOURNAL = {Proceedings of the Royal Society. London. Series A.
              Mathematical, Physical and Engineering Sciences},
    VOLUME = {453},
      YEAR = {1997},
    NUMBER = {1957},
     PAGES = {403--415},
      ISSN = {0962-8444,2053-9169},
   MRCLASS = {58E20},
  MRNUMBER = {1436833},
MRREVIEWER = {Anne-Jo\"elle\ Vanderwinden},
       DOI = {10.1098/rspa.1997.0023},
       URL = {https://doi.org/10.1098/rspa.1997.0023},
}

@incollection {MR2389639,
    AUTHOR = {H\'elein, Fr\'ed\'eric and Wood, John C.},
     TITLE = {Harmonic maps},
 BOOKTITLE = {Handbook of global analysis},
     PAGES = {417--491, 1213},
 PUBLISHER = {Elsevier Sci. B. V., Amsterdam},
      YEAR = {2008},
      ISBN = {978-0-444-52833-9},
   MRCLASS = {58E20 (53C28 53C43 58-02)},
  MRNUMBER = {2389639},
MRREVIEWER = {Andreas\ Gastel},
       DOI = {10.1016/B978-044452833-9.50009-7},
       URL = {https://doi.org/10.1016/B978-044452833-9.50009-7},
}

@misc{BS1,
      title={Stable proper biharmonic maps in Euclidean spheres}, 
      author={Volker Branding and Anna Siffert},
      year={2025},
      eprint={2507.06708},
      archivePrefix={arXiv},
      primaryClass={math.DG},
      url={https://arxiv.org/abs/2507.06708}, 
   JOURNAL = {to appear in Annali della Scuola Normale Superiore di Pisa,
Classe di Scienze},
}

@article {MR4927650,
    AUTHOR = {Branding, Volker and Siffert, Anna},
     TITLE = {Infinite families of harmonic self-maps of ellipsoids in all
              dimensions},
   JOURNAL = {Nonlinear Anal.},
  FJOURNAL = {Nonlinear Analysis. Theory, Methods \& Applications. An
              International Multidisciplinary Journal},
    VOLUME = {261},
      YEAR = {2025},
     PAGES = {Paper No. 113874, 11},
      ISSN = {0362-546X,1873-5215},
   MRCLASS = {58E20 (53C43)},
  MRNUMBER = {4927650},
       DOI = {10.1016/j.na.2025.113874},
       URL = {https://doi.org/10.1016/j.na.2025.113874},
}

@misc{VZ1,
      title={Initial value problems on cohomogeneity one manifolds, I}, 
      author={Luigi Verdiani and Wolfgang Ziller},
      year={2024},
      eprint={2412.06058},
      archivePrefix={arXiv},
      primaryClass={math.DG},
      url={https://arxiv.org/abs/2412.06058}, 
}

@article {MR3427685,
    AUTHOR = {Siffert, Anna},
     TITLE = {Infinite families of harmonic self-maps of spheres},
   JOURNAL = {J. Differential Equations},
  FJOURNAL = {Journal of Differential Equations},
    VOLUME = {260},
      YEAR = {2016},
    NUMBER = {3},
     PAGES = {2898--2925},
      ISSN = {0022-0396,1090-2732},
   MRCLASS = {58E20 (34B15 55M25)},
  MRNUMBER = {3427685},
MRREVIEWER = {Giandomenico\ Orlandi},
       DOI = {10.1016/j.jde.2015.10.023},
       URL = {https://doi.org/10.1016/j.jde.2015.10.023},
}

@article {MR3745872,
    AUTHOR = {Siffert, Anna},
     TITLE = {Harmonic self-maps of {${\rm SU}(3)$}},
   JOURNAL = {J. Geom. Anal.},
  FJOURNAL = {Journal of Geometric Analysis},
    VOLUME = {28},
      YEAR = {2018},
    NUMBER = {1},
     PAGES = {587--605},
      ISSN = {1050-6926,1559-002X},
   MRCLASS = {58E20 (34B15 55M25)},
  MRNUMBER = {3745872},
MRREVIEWER = {Song-il\ Ri},
       DOI = {10.1007/s12220-017-9833-0},
       URL = {https://doi.org/10.1007/s12220-017-9833-0},
}

@article {MR2480860,
    AUTHOR = {P\"{u}ttmann, Thomas},
     TITLE = {Cohomogeneity one manifolds and self-maps of nontrivial
              degree},
   JOURNAL = {Transform. Groups},
  FJOURNAL = {Transformation Groups},
    VOLUME = {14},
      YEAR = {2009},
    NUMBER = {1},
     PAGES = {225--247},
      ISSN = {1083-4362},
   MRCLASS = {57S15 (55M25)},
  MRNUMBER = {2480860},
MRREVIEWER = {Sergio Console},
       DOI = {10.1007/s00031-008-9037-6},
       URL = {https://doi.org/10.1007/s00031-008-9037-6},
}

@article {MR1923478,
    AUTHOR = {Grove, Karsten and Ziller, Wolfgang},
     TITLE = {Cohomogeneity one manifolds with positive {R}icci curvature},
   JOURNAL = {Invent. Math.},
  FJOURNAL = {Inventiones Mathematicae},
    VOLUME = {149},
      YEAR = {2002},
    NUMBER = {3},
     PAGES = {619--646},
      ISSN = {0020-9910,1432-1297},
   MRCLASS = {53C21},
  MRNUMBER = {1923478},
MRREVIEWER = {Burkhard\ Wilking},
       DOI = {10.1007/s002220200225},
       URL = {https://doi.org/10.1007/s002220200225},
}

@article {MR4400726,
    AUTHOR = {Verdiani, L. and Ziller, W.},
     TITLE = {Smoothness conditions in cohomogeneity one manifolds},
   JOURNAL = {Transform. Groups},
  FJOURNAL = {Transformation Groups},
    VOLUME = {27},
      YEAR = {2022},
    NUMBER = {1},
     PAGES = {311--342},
      ISSN = {1083-4362,1531-586X},
   MRCLASS = {53C30},
  MRNUMBER = {4400726},
MRREVIEWER = {Reza\ Mirzaie},
       DOI = {10.1007/s00031-020-09618-9},
       URL = {https://doi.org/10.1007/s00031-020-09618-9},
}

@incollection {MR1155662,
    AUTHOR = {Alekseevski\u i, A. V. and Alekseevski\u i, D. V.},
     TITLE = {{$G$}-manifolds with one-dimensional orbit space},
 BOOKTITLE = {Lie groups, their discrete subgroups, and invariant theory},
    SERIES = {Adv. Soviet Math.},
    VOLUME = {8},
     PAGES = {1--31},
 PUBLISHER = {Amer. Math. Soc., Providence, RI},
      YEAR = {1992},
      ISBN = {0-8218-4107-6},
   MRCLASS = {57S15},
  MRNUMBER = {1155662},
MRREVIEWER = {Yoshinobu\ Kamishima},
       DOI = {10.1007/bf01084048},
       URL = {https://doi.org/10.1007/bf01084048},
}

@book {MR3362465,
    AUTHOR = {Alexandrino, Marcos M. and Bettiol, Renato G.},
     TITLE = {Lie groups and geometric aspects of isometric actions},
 PUBLISHER = {Springer, Cham},
      YEAR = {2015},
     PAGES = {x+213},
      ISBN = {978-3-319-16612-4; 978-3-319-16613-1},
   MRCLASS = {22-01 (22E46 22F05 53-01)},
  MRNUMBER = {3362465},
MRREVIEWER = {Vladimir\ V.\ Gorbatsevich},
       DOI = {10.1007/978-3-319-16613-1},
       URL = {https://doi.org/10.1007/978-3-319-16613-1},
}

@article {MR1758585,
    AUTHOR = {Eschenburg, J.-H. and Wang, McKenzie Y.},
     TITLE = {The initial value problem for cohomogeneity one {E}instein
              metrics},
   JOURNAL = {J. Geom. Anal.},
  FJOURNAL = {The Journal of Geometric Analysis},
    VOLUME = {10},
      YEAR = {2000},
    NUMBER = {1},
     PAGES = {109--137},
      ISSN = {1050-6926,1559-002X},
   MRCLASS = {53C25 (34A12 34C60 34E05 53C30)},
  MRNUMBER = {1758585},
MRREVIEWER = {Megan\ M.\ Kerr},
       DOI = {10.1007/BF02921808},
       URL = {https://doi.org/10.1007/BF02921808},
}

@article {MR85460,
    AUTHOR = {Mostert, Paul S.},
     TITLE = {On a compact {L}ie group acting on a manifold},
   JOURNAL = {Ann. of Math. (2)},
  FJOURNAL = {Annals of Mathematics. Second Series},
    VOLUME = {65},
      YEAR = {1957},
     PAGES = {447--455},
      ISSN = {0003-486X},
   MRCLASS = {17.0X},
  MRNUMBER = {85460},
MRREVIEWER = {A.\ Shields},
       DOI = {10.2307/1970056},
       URL = {https://doi.org/10.2307/1970056},
}

@article {MR95897,
    AUTHOR = {Mostert, Paul S.},
     TITLE = {Errata, ``{O}n a compact {L}ie group acting on a manifold''},
   JOURNAL = {Ann. of Math. (2)},
  FJOURNAL = {Annals of Mathematics. Second Series},
    VOLUME = {66},
      YEAR = {1957},
     PAGES = {589},
      ISSN = {0003-486X},
   MRCLASS = {22.00},
  MRNUMBER = {95897},
MRREVIEWER = {A.\ Shields},
       DOI = {10.2307/1969911},
       URL = {https://doi.org/10.2307/1969911},
}

@article {MR4000241,
    AUTHOR = {P\"{u}ttmann, Thomas and Siffert, Anna},
     TITLE = {Harmonic self-maps of cohomogeneity one manifolds},
   JOURNAL = {Math. Ann.},
  FJOURNAL = {Mathematische Annalen},
    VOLUME = {375},
      YEAR = {2019},
    NUMBER = {1-2},
     PAGES = {247--282},
      ISSN = {0025-5831},
   MRCLASS = {58E20 (34B15 55M25 57S15)},
  MRNUMBER = {4000241},
MRREVIEWER = {Andreas Gastel},
       DOI = {10.1007/s00208-019-01848-x},
       URL = {https://doi.org/10.1007/s00208-019-01848-x},
}

@book {MR2044031,
    AUTHOR = {Baird, Paul and Wood, John C.},
     TITLE = {Harmonic morphisms between {R}iemannian manifolds},
    SERIES = {London Mathematical Society Monographs. New Series},
    VOLUME = {29},
 PUBLISHER = {The Clarendon Press, Oxford University Press, Oxford},
      YEAR = {2003},
     PAGES = {xvi+520},
      ISBN = {0-19-850362-8},
   MRCLASS = {53C43 (53C21 58E20)},
  MRNUMBER = {2044031},
MRREVIEWER = {Eric Loubeau},
       DOI = {10.1093/acprof:oso/9780198503620.001.0001},
       URL = {https://doi.org/10.1093/acprof:oso/9780198503620.001.0001},
}

\end{document}